\documentclass[smallcondensed]{svjour3}     % onecolumn (ditto)
\smartqed  % flush right qed marks, e.g. at end of proof
\usepackage{amssymb}

\usepackage{graphicx}
\usepackage{float}
\usepackage{color}
\usepackage{comment}
\usepackage{mathrsfs}

\usepackage{amsmath}
\begin{document}

\title{A Mathematical Proof of the Four-Color Conjecture (1): Transformation Step%\thanks{Grants or other notes
}


\author{Jin Xu         %\and
        %Second Author %etc.
}


\institute{J. Xu \at
              Key Laboratory of High Confidence Software Technologies (Peking University), Ministry of Education, China \\
              School of Electronics Engineering and Computer Science, Peking University, Beijing 100871, China\\
             % \email{fauthor@example.com}           %  \\
         %  \and
          % S. Author \at
            %  second address
}

\date{Received: date / Accepted: date}

\maketitle

\begin{abstract}

The four-color conjecture has puzzled mathematicians for over 170 years and has yet to be proven by purely mathematical methods. This series of articles provides a purely mathematical proof of the four-color conjecture, consisting of two parts: the transformation step and the decycle step. The transformation step uses two innovative tools, contracting and extending operations and unchanged bichromatic cycles, to transform the proof of the four-color conjecture into the decycle problem of 4-base modules. Moreover, the decycle step solves the decycle problem of 4-base modules using two other innovative tools: the color-connected potential and the pocket operations. This article presents the proof of the transformation step.

\keywords{four-color conjecture \and unchanged bichromatic-cycle maximal planar graphs  \and  contracting and extending operation \and 4-base module \and decycle coloring \and  color connected-potential  \and pocket operation}
\end{abstract}

\section{Introduction}
\label{intro}
Maximal planar graphs (MPGs), a class of planar graphs with the maximum number of edges,  play a crucial role in studying the four-color conjecture (FCC). Since the late 1970s, mathematicians have devoted a tremendous amount of effort to examining the characteristics of MPGs, including their structures, constructions, and colorings, in hopes of proving FCC. To learn more about these studies, please refer to \cite{r1,r2,r3,r4,r5,r6,r7,r8,r9,r10,r11}.   However, proving FCC presents numerous challenges, leading to the proposal of new conjectures such as the uniquely four chromatic planar graphs conjecture \cite{r3,r4}, which further enriches the field of maximal planar graph theory \cite{r11}.

\subsection{Structure of MPGs}
Euler's formula is a powerful tool for analyzing the structural features of planar graphs. It is commonly known that this formula has yielded many results related to the structure of planar graphs, the most important of which is the minimum degree $\delta(G)$ of every planar graph $G$ is at most 5. Specifically, $3\leq \delta(G) \leq 5$ when $G$ is an MPG. There are several variations of Euler's formula, based on which many ``discharging'' approaches have been proposed to explore the unavoidability and reducibility of some configurations of MPGs \cite{r5}.  Kempe claimed that the $k$-wheel configuration ($k=3,4,5$) is reducible, and he used the Kempe change to prove FCC  \cite{r1}. However, Heawood pointed out a hole in Kempe's proof for the $5$-wheel  \cite{r6}. This led to an interest in proving the reducibility of this type of configuration. The basic idea was to search exhaustively for unavoidable 5-wheel-based configurations and then attempt to prove that each is reducible. Unfortunately, the number of such configurations was vast, making it impossible to examine the reducibility of each manually. In 1969, Heesch \cite{r5} introduced the brilliant method of ``discharging'' at an academic conference, based on which suitable rules were designed to explore the unavoidable sets of configurations with the help of computer programs. This idea later inspired Haken and Appel, who spent seven years investigating the configurations in more detail. Eventually, in 1976, with the help of Koch and about 1200 hours of a fast mainframe computer, they gave a computer-based proof of FCC \cite{r7,r8}. They used 487 discharging rules and 1936 unavoidable configurations. Appel and Haken's work opened an avenue for logical reasoning using computers.

There have been concerns among scholars regarding the dependability of computer-assisted proofs for FCC. Even Haken and Appel made several revisions to their work before publishing their final version in 1976. They analyzed 1,936 reducible unavoidable configurations. In 1996, Robertson et al. \cite{r9} utilized a similar method to offer an improved proof of FCC. They were successful in decreasing the number of reducible configurations in the unavoidable set to 633. Despite these advances, mathematicians are still hopeful for a rigorous and concise mathematical proof for FCC, which has not been presented since 1852.

\subsection{Construction of MPGs}

To construct MPGs, in 1891, Eberhard \cite{r13} introduced the concept of a \textbf{pure chord-cycle}, i.e., a cycle in an MPG that contains no vertices in its interior. He proposed a generation system $<K_4: \{\phi_1,\phi_2,\phi_3\}>$ based on three operators $\phi_1,\phi_2,\phi_3$ derived from pure chord-cycle (as shown in Figure \ref{newfig1-1}), that can construct all MPGs from an initial graph $K_4$ (i.e., a complete graph of order 4).

\begin{figure}[H]
	\centering	
\includegraphics[width=12cm]{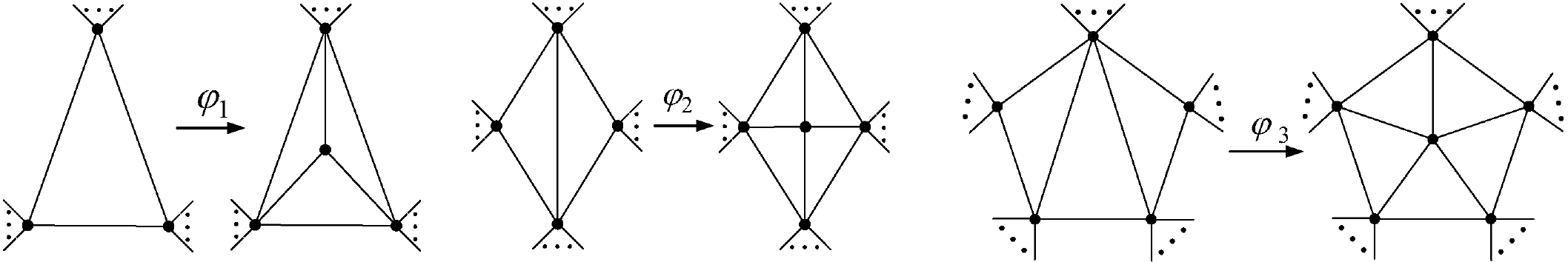}\\
%(a) \hspace{6cm}  (b)    \\
	\caption{The three operators proposed by Eberhard} \label{newfig1-1}
\end{figure}

In 1936, Wagner  \cite{r14}  proposed the technique of  \textbf{edge-flipping} to achieve the transformation of two MPGs with identical order. This involves a local deformation of an MPG $G$, wherein a diagonal edge $ac$ in a diamond subgraph $abcd$ is replaced with the other edge $bd$, provided that $bd$ is not an edge of $G$. Refer to Figure \ref{newfig1-2} for a visual representation.

\begin{figure}[H]
	\centering	
\includegraphics[width=4.2cm]{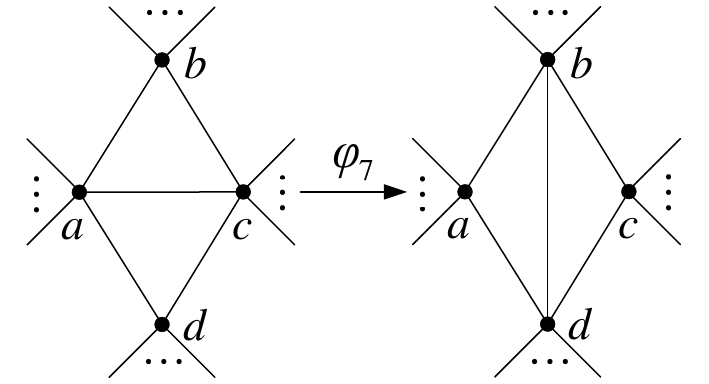}\\
%(a) \hspace{6cm}  (b)    \\
	\caption{Edge-flipping} \label{newfig1-2}
\end{figure}

In 1974, Barnette \cite{r15} and Butler \cite{r16}  independently built a generation system  $<Z_{20}: \{\phi_4,\phi_5,\phi_6\}>$ (see Figure \ref{newfig1-3}) which can generate all 5-connected MPGs, where $Z_{20}$ is the regular dodecahedron.

\begin{figure}[H]
	\centering	
\includegraphics[width=12cm]{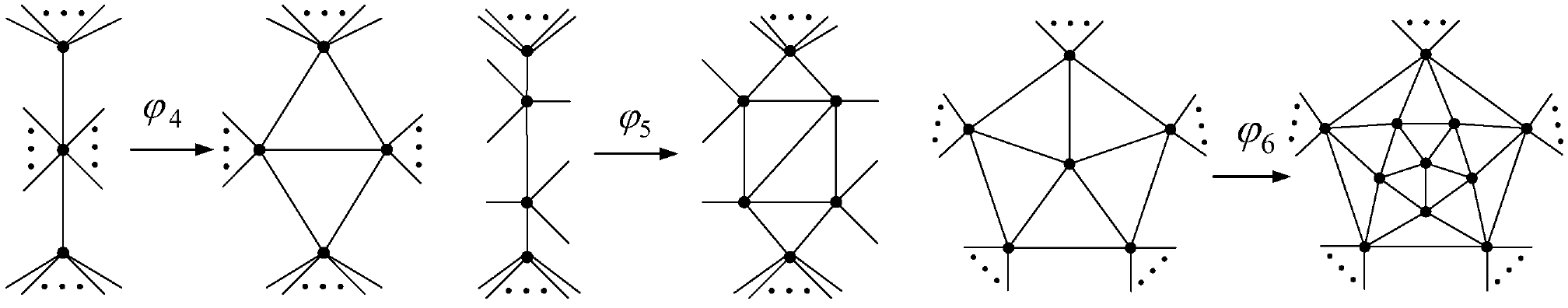}\\
%(a) \hspace{6cm}  (b)    \\
	\caption{The three operators $\phi_4, \phi_5$ and $\phi_6$ proposed by Barnette and Butler} \label{newfig1-3}
\end{figure}

In 1984, Batagelj  \cite{r17} extended the generation system $<Z_{20}; \{\phi_4,\phi_5,\phi_6\}>$ to $<Z_{20}; \{\phi_4,\phi_5,\phi_7\}>$, which can generate 3- and 4-connected MPGs of minimum degree 5. Analogously, he depicted a method to construct (even) MPGs with a minimum degree of at least 4.

In 2016, a new method for constructing MPGs called the \textbf{contracting
and extending operation (CE operation for short)}, is proposed in \cite{r18}.  CE-system has four pairs of operators and one starting graph $K_4$. Compared with the existing approaches, the system can naturally connect the coloring with the construction of MPGs.  In Section \ref{sec:ecsystem}, we will delve into this system and its workings.

\subsection{Colorings of MPGs}

Regarding the colorings of MPGs, the Kempe change (or K-change for short) is a fundamental and essential technique proposed by Kempe in 1879. Its basic function is to induce a new 4-coloring from an existing 4-coloring. Currently, researchers are focused on determining whether the $k$-chromatic graph (where $k \geq 3$) is a Kempe graph - meaning that its $k$-colorings can be generated from a given $k$-coloring of the graph using K-changes. For more information on K-changes, please refer to the review \cite{r19}.

There exist numerous MPGs that are non-Kempe. For instance, Figure \ref{newfig2-1} displays three 4-colorings, denoted as $f_1$, $f_2$, and $f_3$, where $f_3$ cannot be obtained from $f_1$ or $f_2$ via K-change. In this paper, we identify a group of non-Kempe MPGs called unchanged bichromatic-cycle maximal planar graphs (UBCMPGs) that possess noteworthy characteristics. We introduce the concept of 4-base modules based on UBCMPGs, and prove that the four-color problem can be transformed into the decycle problem of 4-base modules using CE operations. We also tackle this problem by introducing two innovative tools, the color-connected potential and the pocket operations, proving that every 4-base module contains a decycle coloring.

\section{Preliminaries}
\label{sec:1}

All graphs analyzed in this paper are finite, simple, and undirected. We use $V(G)$ and $E(G)$ to denote the \emph{vertex set} and the \emph{edge set} of a graph $G$, respectively. The \emph{order} of  $G$ refers to the number of vertices in $V(G)$. If there is a vertex $v\in V(G)$ with an adjacent vertex $u\in V(G)$, we call $u$ a \emph{neighbor} of $v$. The set of all neighbors of $v$ is referred to as the \emph{neighborhood} of $v$ in $G$ and is denoted by $N_G(v)$. Additionally, we define $N_G[v]=N_G(v)\cup \{v\}$, which is known as the \emph{closed neighborhood} of $v$ in $G$. The \emph{degree} of a vertex $v$ in $G$ is denoted by $d_G(v)$ and is defined as the cardinality of $N_G(v)$, i.e., $d_G(v)=|N_G(u)|$. When there is no ambiguity, we will use $V, E, d(v), N(v), \delta$, and $\Delta$ instead of $V(G)$, $E(G)$, $d_G(v)$, $N_G(v)$, $\delta(G)$, and $\Delta(G)$, respectively. Furthermore, we use $\delta(G)$ and $\Delta(G)$ to represent the minimum and maximum degree of $G$, respectively. A graph $H$ is called a \emph{subgraph} of $G$ if $V(H) \subseteq V(G)$ and $E(H)\subseteq E(G)$; furthermore, if two vertices $x$ and $y$ are connected by an edge $xy$ in $H$ if and only if $xy\in E(G)$, then $H$ is called a \emph{subgraph induced by $V(H)$}, denoted by $G[V(H)]$. By combining two vertex-disjoint graphs $G$ and $H$ and adding edges to connect every vertex of $G$ to every vertex of $H$, one obtains the \emph{join} of $G$ and $H$, denoted by $G\vee H$. We denote by $K_n$ and $C_n$  the complete graph and cycle of order $n$, respectively.
The join $C_{n}\vee K_{1}$ is called a \emph{n-wheel with $n$ spokes}, denoted by $W_{n}$, where $C_n$ and $K_1$ are called the \emph{cycle} and \emph{center} of  $W_{n}$, respectively. Let $V(K_1)=\{x\}$ and  $C_n=x_1x_2\ldots x_nx_1$. Then, $W_n$ can be represented as  $x$-$x_1x_2\ldots x_nx_1$.  The \emph{length} of a path or a cycle is the number of edges the path or the cycle contains.  We call a path (or cycle) an \emph{$\ell$-path} (or \emph{$\ell$-cycle}) if its length is equal to $\ell$.

A \emph{planar graph} is a graph that can be drawn on a plane in a way that edges only meet at their common ends. This drawing is called a \emph{plane graph} or \emph{planar embedding} of the graph. In this paper, any planar graph we consider will refer to one of its planar embeddings. A \emph{maximal planar graph} (MPG) is a planar graph to which no new edges can be added without violating its planarity. Additionally, a \emph{triangulation} is a planar graph where every face is bounded by three edges, including the infinite face. It can be easily demonstrated that an MPG is equivalent to a triangulation.

Let $G$ be a planar graph that has a boundary $C$ on its outer face. If $C$ contains at least four edges and all of its other faces are triangles, then $G$ is referred to as a \emph{semi-maximal planar graph with respect to $C$}, or a \emph{SMPG} for short, where $C$ is called the \emph{outer-cycle} of $G^C$. SMPGs are also known as \emph{configurations}. Figure \ref{configurations} displays the 55-configuration and 56-configuration, two well-known SMPGs.

 \begin{figure}[H]
	\centering	
\includegraphics[width=6cm]{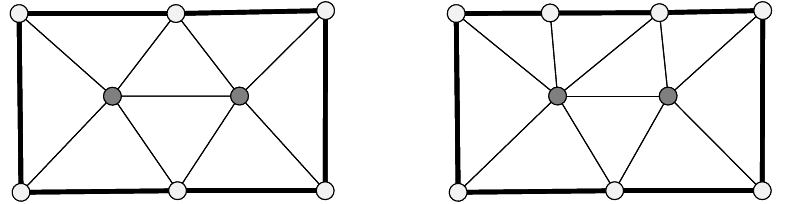}\\
%(a) \hspace{6cm}  (b)    \\
	\caption{55-configuration (left) and  56-configuration (right)} \label{configurations}
\end{figure}

\begin{comment}
Let $G$ be an arbitrary minimum counterexample to FCC in terms of $V(G)$; that is,  $G$ is a MPG, $G$ is not 4-colorable, and every $n$-order MPG such that $n<|V(G)|$ is 4-colorable. If a configuration $G^C$ is not contained in $G$, then $G^C$ is said to be \emph{reducible}. This paper aims to prove that the two configurations, i.e., 55- and 56-configurations, shown in Figure \ref{configurations} are reducible.
\end{comment}

A \emph{$k$-vertex-coloring}, or simply a $k$-coloring, of a
 graph \emph{G} is a mapping $f$ from  $V(G)$ to the color sets $C(k)=\{1,2,\ldots,k\}$ such
 that $f(x)\neq f(y)$ if $xy\in E(G)$.  A graph $G$ is \emph{$k$-colorable} if it has a $k$-coloring. The minimum $k$ for which a graph $G$ is $k$-colorable is called the \emph{chromatic number} of $G$, denoted by $\chi(G)$. If $\chi(G)=k$, then  $G$ is called a \emph{$k$-chromatic graph}.
Alternatively,
 each $k$-coloring $f$ of  $G$ can be viewed as a partition  $\{V_{1}, V_{2},\cdots, V_{k}\}$ of $V$,
 where $V_{i}$, called the \emph{color classes} of $f$, is an independent set (every two vertices in the set are not adjacent). %and denotes the set of vertices assigned color $i$.
 Clearly, such a partition is unique. So it can be written as $f=(V_{1},V_{2},\cdots,
 V_{k})$. 
%The set of all $k$-colorings of a graph $G$ is denoted by $C_{k}(G)$. 
Two $k$-colorings $f=(V_{1},V_{2},\cdots,
 V_{k})$ and $f'=(V'_{1},V'_{2},\cdots, V'_{k})$ of $G$ are \emph{equivalent} if there exists a permutation $\sigma$ on \{1,2,\ldots, k\} such that $\sigma(i)=j_i$ and  $V_i=V'_{j_i}$ for $i=1,2,\ldots, k$. The set of $k$-colorings that are equivalent to $f$ is called \emph{$f$-equivalence class}. It is easy to see that the $f$-equivalence class contains $k!$ $k$-colorings.
 For a $k$-chromatic graph $G$, we use $C^{0}_{k}(G)$ to
    denote the set of $k$-colorings of $G$ such that no two colorings belong to the same equivalence class and the $i$th color class of each  $k$-coloring is colored with $i$.

{\bf Remark 1}   In the absence of any specific notation, we designate the color set as \{1, 2, 3, 4\} for the 4-colorings of  MPGs and SMPGs.

Let $G$ be a $k$($\geq 3$)-chromatic graph  and  $H$  a subgraph of $G$. For any $g\in C_k^0(G)$,
we use $g(H)$ to denote the set of colors assigned to $V(H)$ under $g$. If $|g(H)|=2$, then we call $H$ a \emph{bichromatic subgraph of $G$ under $g$}; in particular, when $H$ is a cycle (or path),  we call it a \emph{bichromatic cycle} (or \emph{bichromatic path}) of $g$. Let $h\in C_k^0(H)$. If $h(v)=g(v)$ for every $v\in V(H)$, then we call $h$ the \emph{restricted coloring of $g$ to $H$} and call $g$ an \emph{extended coloring of $h$ to $G$}.

%we use $g(H)$ to denote the set of colors assigned to $V(H)$ under $g$. Let
% We refer to the coloring of $g$ restricted to $V(H)$, denoted by $h$,  as the \emph{restricted coloring of %$g$ to $H$}, and call $g$  an \emph{extended coloring of $h$ to $G$}. If a bichromatic subgraph of $G$ under %$g$ is a cycle (or path), then we call it  a \emph{bichromatic cycle} (or \emph{bichromatic path}) of $g$.

Let $f$ be a 4-coloring of a 4-chromatic MPG (or  SMPG) $G$.  If $f$  does not contain any bichromatic cycle, then  $f$ is called a \emph{tree-coloring} of $G$ and $G$ is \emph{tree-colorable}; otherwise, $f$ is a \emph{cycle-coloring} and $G$ is \emph{cycle-colorable} \cite{r12,r20,r21}.

Given a $k$-chromatic  graph $G$ and a $k$-coloring $f$ of $G$, we use $G_{ij}^f$ to denote the subgraphs induced by vertices colored with $i$ and $j$ under $f$, and use $\omega(G_{ij}^f)$  to denote the number of components of $G_{ij}^f$, $i\neq j$. If $f$ is fixed, we can omit the mark $f$ in $G_{ij}^f$. The components of $G_{ij}$ are called  \emph{$ij$-components} of $f$; particularly, we refer to an $ij$-component as an \emph{$ij$-path} of $f$ and \emph{$ij$-cycle} of $f$ if it is a path and a cycle, respectively.  Let
\[
\omega(f)=\sum\limits_{1\leq i<j\leq k} \omega(G_{ij}^f).
\]
When $\omega(G_{ij}^f)\geq 2$, the \emph{Kempe-change} (or \emph{K-change} for short) on an $ij$-component of $f$ is to interchange colors $i$ and $j$ of vertices in the $ij$-component.

Let $G$ be a cycle-colorable MPG (or SMPG) and $f$  a cycle-coloring of $G$. Suppose that $C$ is a bichromatic cycle of $f$ with $f(C)$=$\{i,j\}$, where $i\neq j$ and $i,j\in \{1,2,3,4\}$. The \emph{$\sigma$-operation of $f$ with respect to $C$}, denoted as $\sigma(f, C)$, is to interchange the colors $s$ and $t$ ($\{s,t\}=\{1,2,3,4\}\setminus \{i,j\}$) of vertices in the interior (or exterior) of $C$. Clearly, $\sigma(f,C)$  transform  $f$ into a new cycle-coloring of $G$, denoted by $f^c$. We refer to $f^c$ and $f$ as a pair of \emph{complement colorings (with respect to $C$)}.

 {\bf Remark 2} The $\sigma$-operation is essentially an extension of the $K$-change, which in turn comprises one or more $K$-changes. It's worth noting that the two colorings that result from applying $\sigma(f, C)$--which involve interchanging vertex colors within and outside of $C$--belong to the same equivalence class. As such, we can always assume that the colors of interior vertices of $C$ are swapped when we carry out a $\sigma(f, C)$ operation.

Let $G$ be a 4-chromatic MPG. If a 4-coloring $f_0$ can be obtained from $f$ by a sequence of $K$-changes, then  we say that $f$ and $f_0$ are \emph{Kempe-equivalent}. For an arbitrary 4-coloring $f\in C_4^0(G)$, we refer to
 \begin{center}
 $F^{f}(G)=\{f_0; f_0$ and $f$ are Kempe-equivalent, $f_0\in C_4^0(G)$\}
\end{center}
as \emph{the Kempe-equivalence class of $f$}. Furthermore, if  $F^{f}(G)=C_4^0(G)$, then $G$ is called a \emph{Kempe} graph; otherwise, a \emph{non-Kempe} graph.

For terms and notations not defined here, readers can refer to \cite{r22,r10,r11,r12}.

\section{Unchanged Bichromatic-Cycle Maximal Planar Graphs}
This section is aimed at introducing the unchanged bichromatic-cycle MPGs. It will cover the definitions, as well as the properties and classifications associated with them.

\subsection{Definitions and properties}
Let $G$ be a 4-chromatic MPG  with $\delta(G)\geq 4$, $f\in C_4^0(G)$, and $C\in C^2(f)$, where $C^2(f)$ represents the set of bichromatic cycles of $f$. If $|f'(C)|$=2 holds for any $f'\in F^f(G)$, then we call $C$  an \emph{unchanged bichromatic-cycle} (abbreviated by UB-cycle) of $f$ (also of $G$), and $f$ an \emph{unchanged bichromatic-cycle coloring with respect to $C$} (abbreviated by UBC-coloring with respect to $C$) of $G$. Also, $G$ is called an \emph{unchanged bichromatic-cycle MPG  with respect to $C$} (abbreviated by \emph{UBCMPG} with respect to $C$). The set of bichromatic cycles of all colorings belonging to $F^f$ is called the \emph{Kempe cycle-set on $f$}, denoted by $C^2(F^f)$, i.e.

\begin{equation}
C^2(F^f)=\bigcup \limits _{f'\in F^f(G)}C^2(f')
\end{equation}

Consider a graph $G$ with two cycles $C_1$ and $C_2$. We say that $C_1$ and $C_2$ are \emph{intersecting} if there exist vertices in the interior of $C_2$ and $C_1$ respectively that belong to $V(C_2)$ and $V(C_1)$. Otherwise, if there are no such vertices, we say that $C_1$ and $C_2$ are \emph{nonintersecting}.

Figure \ref{newfig2-1} $(a)\sim (c)$ depict three 4-colorings $f_1$, $f_2$, and $f_3$ of the minimum UBCMPG. The bold lines in Figures $(a)$ and $(b)$ represent UB-cycles. We observe that $F^{f_1} = \{f_1, f_2\}$ and $F^{f_3} = \{f_3\}$.

\begin{figure}[H]
  \centering
  \includegraphics[width=10cm]{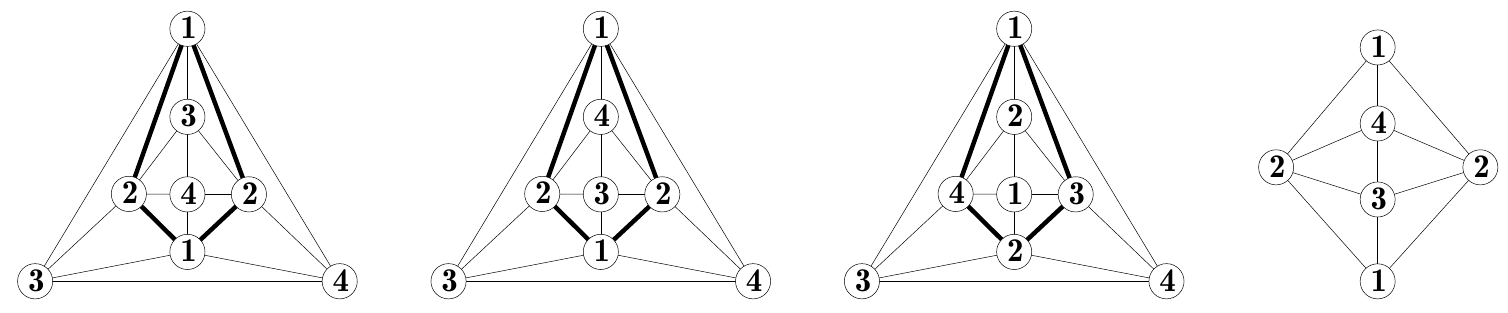}\\
  \hspace{8mm}(a) $f_1$ \hspace{1.8cm}  (b) $f_2$ \hspace{1.8cm}  (c) $f_3$ \hspace{1.4cm} (d) $B^4$
  \caption{The UBCMPG with the minimum order}\label{newfig2-1}
\end{figure}

{\bf Remark 3} A UBCMPG can have multiple UB-cycles.

\begin{theorem}\label{thm2.1}
Let $G$ be a 4-chromatic MPG with $\delta(G)\geq 4$, $f\in C_4^0(G)$, and $C\in C^2(f)$. Then, $C$ is a UB-cycle of $f$ if and only if for any $C'\in C^2(F^f)$ with $f(C')\neq f(C)$, $C$ and $C'$ are nonintersecting.
\end{theorem}
\begin{proof}
(Necessity). Suppose, to the contrary, that there is a cycle $C'\in C^2(F^f)$ with $f(C')\neq f(C)$ such that  $C$ and $C'$ are intersecting. Then, $C'$ is a bichromatic cycle of some coloring in  $F^f(G)$, say $f'$, i.e., $C'\in C^2(f')$. Since $C$ is a UB-cycle of $f$, it follows that $C\in C^2(f')$. Let $f''=\sigma(f',C')$. Since  $C$ and $C'$ are intersecting, we have that $C\notin C^2(f'')$,  a contradiction.

(Sufficiency). Observe that every $f'\in F^f(G)$ is obtained by implementing $\sigma$-operation of a 4-coloring $g\in F^f(G)$ with respect to a bichromatic cycle of $g$. We have that $C\in C^2(f')$ for every $f'\in F^f(G)$, since $C$ and $C'$ are nonintersecting for every $C'\in C^2(F^f)$. \qed
\end{proof}

The following result holds directly from Theorem \ref{thm2.1}.

\begin{corollary}\label{coro2.2}
Let $G$ be a UBCMPG with respect to $C$. Then, every vertex on $C$ has degree at least 5 in $G$.
\end{corollary}

\subsection{Types}

Let $G$ be a 4-chromatic maximal planar graph with $\delta(G)\geq 4$, and $f\in C_4^0(G)$ be a cycle-coloring. If $C^2(f)$ contains no UB-cycle, then we call $f$ a \emph{cyclic cycle-coloring}. Let $C_{4U}^0(G)$ denote the set of UBC-colorings of $G$,  $C_{4T}^0(G)$ denote the set of tree-colorings of $G$, and  $C_{4C}^0(G)$ denote the set of cyclic cycle-colorings of $G$. It is evident that $C_{4U}^0(G)\cap C_{4T}^0(G)=\emptyset, C_{4U}^0(G)\cap C_{4C}^0(G)=\emptyset$, $C_{4T}^0(G) \cap C_{4C}^0(G)=\emptyset$, and
\begin{equation}\label{equadd-1}
C_4^0(G)=C_{4U}^0(G)\cup C_{4T}^0(G)\cup C_{4C}^0(G)
\end{equation}

Based on Equation (\ref{equadd-1}), we divide UBCMPGs $G$ into the following types.

\begin{itemize}
 \item {Pure-type}: $C_4^0(G)=C_{4U}^0(G)$;
\vspace{0.1cm}

 \item{Tree-type}: $C_4^0(G)=C_{4U}^0(G)\cup C_{4T}^0(G)$;
\vspace{0.1cm}

 \item{Cycle-type}: $C_4^0(G)=C_{4U}^0(G)\cup C_{4C}^0(G)$;
\vspace{0.1cm}

 \item{Hybrid-type}: $C_4^0(G)=C_{4U}^0(G)\cup C_{4T}^0(G)\cup C_{4C}^0(G)$.
\end{itemize}

It is important to note that each of the above four types has corresponding graphs. The 17-order graph $G$, as presented in Figure \ref{fig4}, is a pure-type UBCMPG, where $C_4^0(G)$ includes eight pairs of complement UBC-colorings. Concerning pure-type UBCMPGs, we put forward a conjecture that requires further exploration.

\begin{conjecture}\label{conjecture1}
An MPG is a pure-type UBCMPG if and only if it is the graph shown in Figure \ref{fig4}.
\end{conjecture}

\vspace{-0.5cm}
 \begin{figure}[H]
  \centering
  \includegraphics[width=12.5cm]{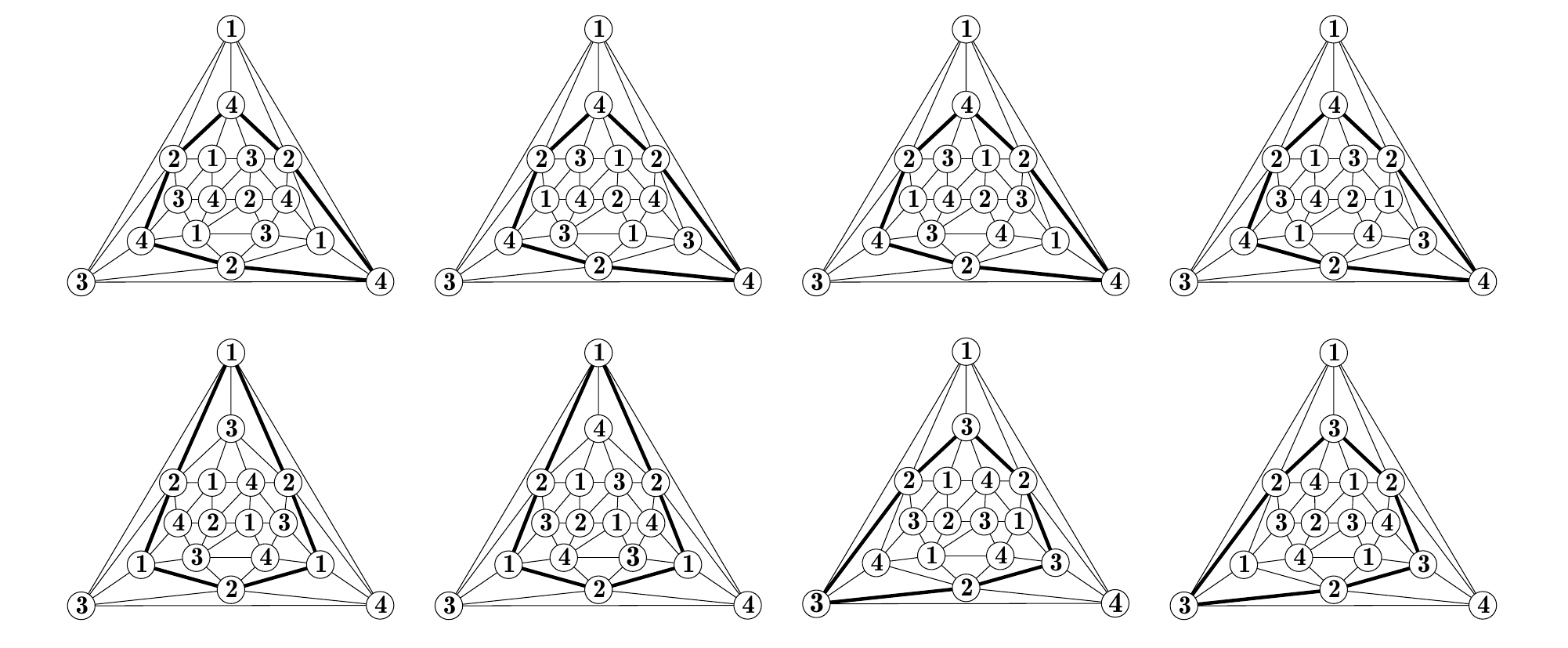}\\
  \includegraphics[width=12.5cm]{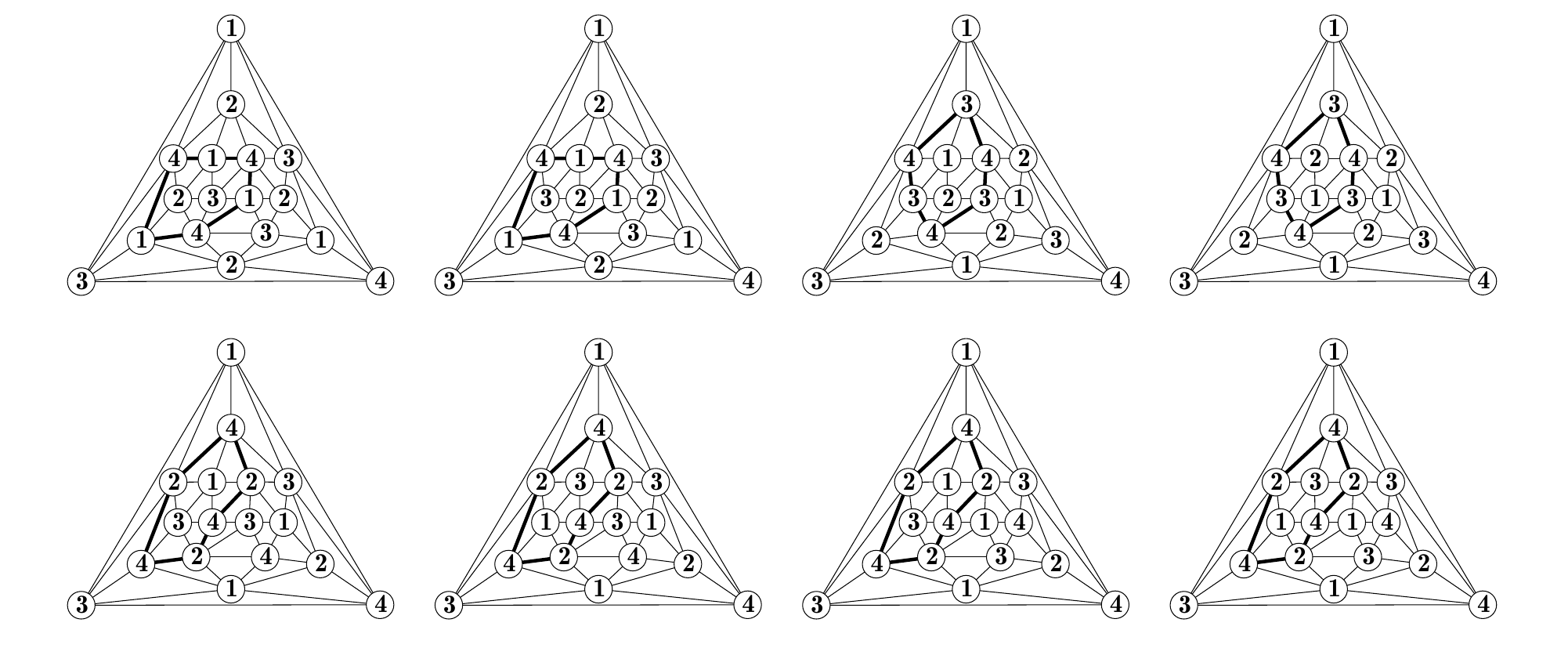}\\
  \caption{A pure-type UBCMPG and all its 4-colorings}\label{fig4}
\end{figure}

The graph shown in Figure \ref{newfig2-1} is a tree-type UBCMPG of order 8, which contains one pair of complement UBC-colorings and one tree-coloring. The graph shown in Figure \ref{fig5} is a tree-type UBCMPG of order 12, which contains two pairs of complement UBC-colorings and two tree-colorings.

{\bf Remark 4} For any $n\geq 2$, there is  a  tree-type UBCMPG $G$ of order $4n$ and $|C_4^0(G)|=2^{n-1}+2^{n-2}$, where the numbers of tree-colorings and UBC-colorings are $2^{n-2}$ and $2^{n-1}$, respectively \cite{r23}.

\begin{figure}[H]
  \centering
 % \includegraphics[width=15cm]{4-1}\\
  \includegraphics[width=12.4cm]{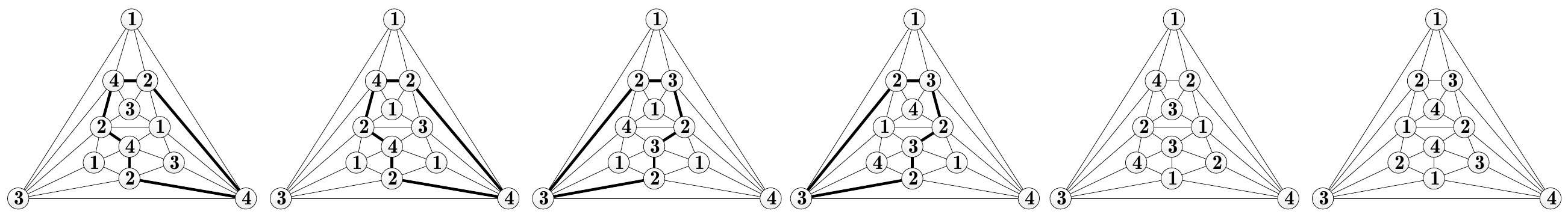}\\
  \caption{A tree-type UBCMPG of order 12 and all its 4-colorings}\label{fig5}
\end{figure}

\section{Base Modules}

\subsection{Definitions and Types}\label{sec4-1}
Let $G^C$ be an SMPG with respect to a cycle $C$. If there is another SMPG  $G_1^C$ with respect to $C$ such that $G^C\cap G_1^{C}=C$ and $G=G^C\cup G_1^{C}$ is a UBCMPG with respect to $C$, then we call $G^C$ and $G_1^C$ \emph{base modules with respect to $C$} (or \emph{base modules} for short). Let $f$ be a  UBC-coloring with respect to $C$ of $G$. We refer to  the restricted coloring of $f$ to $G^C$ and  $G_1^{C}$ as a \emph{module coloring} of $G^C$ and  $G_1^{C}$, respectively. The minimum base module (named identity module), denoted by $B^4$, is depicted in Figure \ref{newfig2-1} (d).

A base module $G^C$ with $|V(C)|=4$ is called a \emph{4-base module}.  Note that only 4-base models are used in this paper. Therefore, all the base modules mentioned in the following discussion are 4-base modules. To specifically refer to a 4-base module, we use the notation $G^{C_4}$.   

Three categories of 4-base modules exist based on module colorings. Let $G^{C_4}$ be a 4-base module with respect to 4-cycle $C_4$, and $f$ be a module coloring of $G^{C_4}$.

 \begin{figure}[H]
  \centering
 \includegraphics[width=7cm]{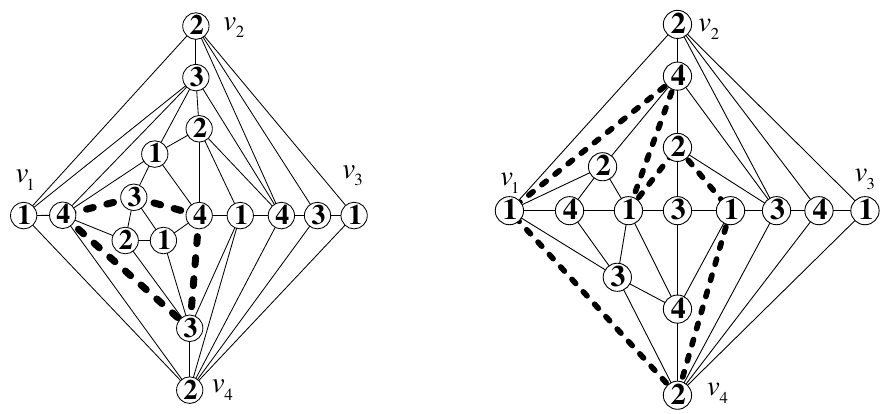}\\
  %\includegraphics[width=11cm]{new1}\\
% \hspace{1cm}(a)  \hspace{3cm} (b) 
  \caption {Examples of cycle-type 4-base module (left) and cyclic-cycle-type 4-base module (right)}\label{new1}
\end{figure}

\begin{itemize}
   \item If $C_4$ is the unique bichromatic cycle of $f$, then $G^{C_4}$ is called a  \emph{tree-type 4-base module};
      % \vspace{0.05cm}
   \item Suppose that $G^{C_4}$ contains a bichromatic cycle $C^{*} (\neq C_4)$ of $f$. If there exists a SMPG $G_1^{C_4}$ (with respect to $C_4$) such that $G^{C_4}\cup G_1^{C_4}$ is a UBCMPG (with respect to $C_4$) and $C^{\star}$ is a UB-cycle of  $G^{C_4}\cup G_1^{C_4}$, then we call $C^{*}$ a UB-cycle of $G^{C_4}$. If $G^{C_4}$ contains a UB-cycle, then  $G^{C_4}$ is  a \emph{cycle-type 4-base module}. The first graph shown in Figure \ref{new1}  depicts a cycle-type 4-base module, which contains a UB-cycle in the interior of $C_4=v_1v_2v_3v_4v_1$ (marked with  dashed bold lines);
        %\vspace{0.05cm}
   \item Let $C'(\neq C_4)$ be a bichromatic cycle of $f$. If $C'$ is not a UB-cycle of $f$, then $C'$ is called a \emph{cyclic cycle} of $G^{C_4}$. If $\mathbb{C}=C^2(F^f(G^{C_4}))\setminus \{C_4\}$ is not empty and contains only cyclic cycles, then  $G^{C_4}$ is called a  \emph{cyclic cycle-type 4-base module}.
   \noindent Since  the subgraph of $G^{C_4}$ induced by the set of vertices belonging to a cycle $C'\in \mathbb{C}$ and its interior is  a SMPG, denoted by $G^{C'}$, it follows that  the union of $G^{C'}$ over all $C'\in \mathbb{C}$ is one or more SMPGs. Observe that each such SMPG $G^C$ is the union of some $G^{C^1}, G^{C^2}, \ldots, G^{C^k}$, where $C^i\in \mathbb{C}$ for $i=1,2,\ldots,k$; that is, $C=C^1\cup C^2\cup \ldots \cup C^k$. We call $G^C$ a \emph{family of cyclic cycles} of $\mathbb{C}$, and  $C$ a \emph{shell} of $G^{C_4}$. The second graph shown in Figure \ref{new1} is a cyclic cycle-type 4-base module,
       in which the shell is marked with dashed bold lines. This 4-base module has only one family of cyclic cycles for this 4-base module.
\end{itemize}

\subsection{Properties of 4-base modules}

Suppose that $G^{C_4}$ is a 4-chromatic SMPG. If no specified note, we use  $F_2(G^{C_4})\subset C_4^0(G^{C_4})$ to denote the set of 4-colorings of $G^{C_4}$ such that $C_4$ is colored with exact two colors, and let $C_4=v_1v_2v_3v_4v_1$.  Whennever $F_2(G^{C_4})$ is not empty,  we assume that $f(v_1)=f(v_3)=1$ and  $f(v_2)=f(v_4)=2$  for any 4-coloring $f\in F_2(G^{C_4})$.  Based on the agreement, we can define $P_{1i}^{f}(v_1,v_3)$ as the set of $1i$-path (under $f$) from $v_1$ to $v_3$, and $P_{2i}^{f}(v_2,v_4)$  the set of $2i$-path (under $f$) from $v_2$ to $v_4$, where $i\in \{3,4\}$. We refer to the paths in $P_{1i}^{f}(v_1,v_3)$ and $P_{2i}^{f}(v_2,v_4)$ for $i=3,4$ as \emph{$ji$-endpoint-paths of $f$}, where $j\in \{1,2\}$, and use  $\ell^{1i}$ and $\ell^{2i}$ to denote a path in $P_{1i}^{f}(v_1,v_3)$ and $P_{2i}^{f}(v_2,v_4)$, respectively. In particular, when $G^{C_4}$ is a 4-base module and $f$ is a module coloring of $G^{C_4}$,  these endpoint-paths are called \emph{module paths of $f$}.  Based on these notations, we have the following result.

\begin{theorem}\label{thm3.3}
Let $G^{C_4}$ be a 4-chromatic SMPG, $C_4=v_1v_2v_3v_4v_1$. If $F_2(G^{C_4})\neq \emptyset$, then for any $f\in F_2(G^{C_4})$, exact two of $P_{13}^{f}(v_1,v_3)$, $P_{14}^{f}(v_1,v_3)$, $P_{23}^{f}(v_2,v_4)$ and $P_{24}^{f}(v_2,v_4)$ are nonempty.
\end{theorem}
\begin{proof}
The conclusion follows directly from the fact that $P_{13}^{f}(v_1,v_3)=\emptyset$ if and only if $P_{24}^{f}(v_2,v_4)\neq \emptyset$, and $P_{14}^{f}(v_1,v_3)=\emptyset$ if and only if $P_{23}^{f}(v_2,v_4)\neq \emptyset$. \qed
\end{proof}

According to Theorem \ref{thm3.3}, the 4-colorings within $F_2(G^{C_4})$ can be divided into two categories: \emph{cross-coloring} and \emph{shared-endpointcoloring}. Let $f\in F_2(G^{C_4})$. If $P^f_{1i}(v_1,v_3)\neq \emptyset$ and $P^f_{2i}(v_2,v_4)\neq \emptyset$ for some $i\in \{3,4\}$, then $f$ is deemed a \emph{cross-coloring} of $G^{C_4}$. Please refer to Figure \ref{fig3-2} (a) for an example of \emph{cross-colorings}. Additionally, if $P^f_{1i}(v_1,v_3)=\emptyset$ for every $i\in \{3,4\}$ or $P^f_{1i}(v_1,v_3)\neq\emptyset$ for every $i\in \{3,4\}$, then $f$ is called a \emph{shared-endpoint coloring} of $G^{C_4}$ on $\{v_2,v_4\}$ (when $P^f_{1i}(v_1,v_3)=\emptyset, i=3,4$) or a \emph{shared-endpoint coloring} of $G^{C_4}$ on $\{v_1,v_3\}$ (when $P^f_{1i}(v_1,v_3)\neq \emptyset, i=3,4$), respectively. For an example of \emph{shared-endpoint colorings} on $\{v_2,v_4\}$, please refer to Figure \ref{fig3-2} (b).

 \begin{figure}[H]
  \centering
 \includegraphics[width=5.5cm]{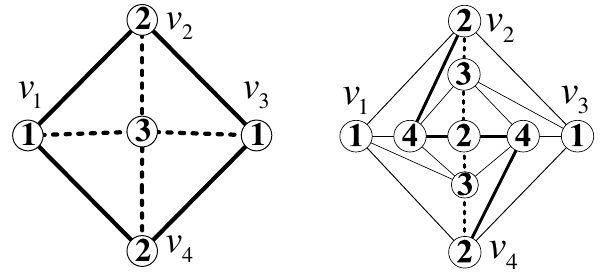}\\
  (a) $f$ \hspace{2cm} (b)  $f_1$
  \caption {Illustration for  cross-colorings and  shared-endpoint colorings}\label{fig3-2}
\end{figure}

\begin{theorem} \label{thm3.4}
Let  $G^{C_4}$ be a 4-chromatic SMPG such that $F_2(G^{C_4})\neq \emptyset$, where $C_4=v_1v_2v_3v_4v_1$. Then, $G^{C_4}$ is a 4-base module if and only if there exists a $f_0\in F_2(G^{C_4})$ such that  $F_2^{f_0}(G^{C_4})$ contains only shared-endpoint colorings on either $\{v_2,v_4\}$ or $\{v_1,v_3\}$, where $F_2^{f_0}(G^{C_4}) \subseteq F^{f_0}(G^{C_4})$ is the set of colorings in $F^{f_0}(G^{C_4})$ such that $C_4$ are colored with two colors.
\end{theorem}

\begin{proof}
(Necessity) Suppose that $G^{C_4}$ is a 4-base module, and let $G_1^{C_4}$ be an arbitrary SMPG with respect to $C_4$ such that $G_1^{C_4}\cap G^{C_4}=C_4$ and $G=G_1^{C_4}\cup G^{C_4}$ is a UBCMPG with respect to $C_4$. Clearly, $F_2(G_1^{C_4})\neq \emptyset$, i.e., there exists a coloring $f_1\in F_2(G_1^{C_4})$ such that $|f_1(C_4)|=2$, where  $f_1(v_1)=f_1(v_3)=1$ and $f_1(v_2)=f_1(v_4)=2$. If for any $f'\in F_2(G^{C_4})$,  $F_2^{f'}(G^{C_4})$ contains a cross-coloring $f''$, then $f''\cup f_1$ or $f''\cup \sigma(f_1, C_4)$ contains a bichromatic cycle of $G$ that is intersect with $C_4$. By Theorem \ref{thm2.1}, $G$ is not a UBCMPG, a contradiction. Therefore, there exists a $f_0\in F_2(G^{C_4})$ such that  $F_2^{f_0}(G^{C_4})$ contains no cross-coloring, i.e., $F_2^{f_0}(G^{C_4})$ contains only   shared-endpoint colorings on $\{v_2,v_4\}$ or $\{v_1,v_3\}$ (by Theorem \ref{thm3.3}). Suppose that there are two colorings $f, f^c\in F_2^{f_0}(G^{C_4})$ such that $f$ is a shared-endpoint coloring  on $\{v_2,v_4\}$ and $f^c$ is a shared-endpoint coloring on $\{v_1,v_3\}$ (see Figure \ref{fig3-3} for an illustration of this case, where $f$ and $f^c$ are complement colorings with respective to $C_{34}$). Analogously, we have that $f_1 \cup f$ or $f_1 \cup f^c$ contains a bichromatic cycle of $G$ that intersects with $C_4$, and also a contradiction.

\begin{figure}[H]
  \centering
 \includegraphics[width=5.5cm]{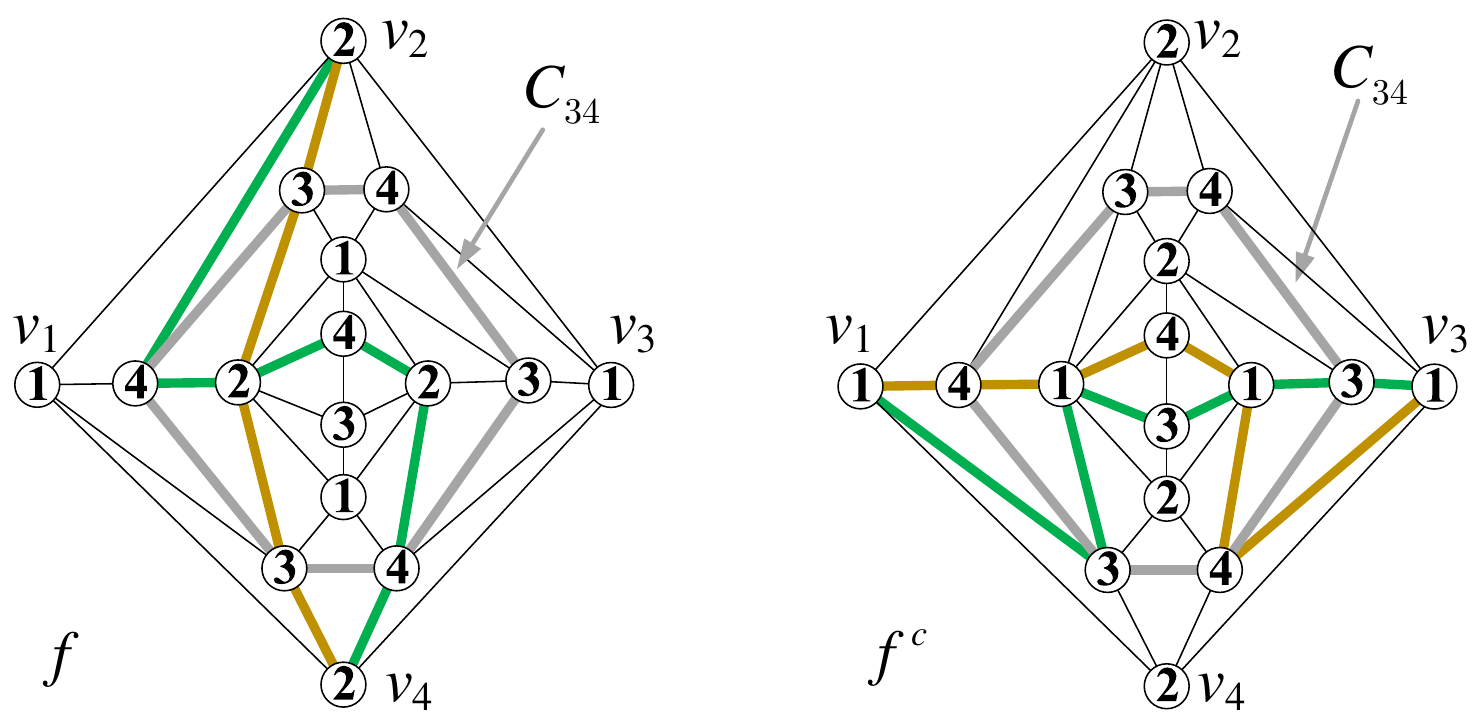}\\
  \caption {An illustration for  the proof of Theorem \ref{thm3.4} (I)}\label{fig3-3}
\end{figure}
\vspace{-1cm}
\begin{figure}[H]
  \centering
 \includegraphics[width=5.5cm]{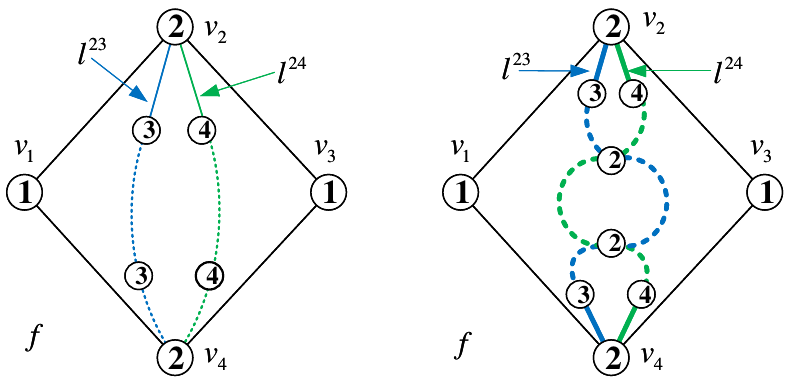}\hspace{1cm}
  \includegraphics[width=2.5cm]{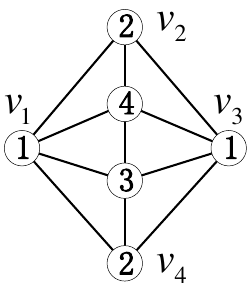}\\
  \hspace{0.5cm}(a) A diagram of shared-endpoint-coloring \hspace{1.5cm} (b)a 4-coloring $g$ of $B^4$
  \caption {An illustration for  the proof of Theorem \ref{thm3.4} (II)}\label{fig3-4}
\end{figure}

(Sufficiency) 
Suppose that $f_0\in F_2(G^{C_4})$ such that $F^{f_0}(G^{C_4})$ contains only shared-endpoint colorings on $\{v_2,v_4\}$ (see Figure \ref{fig3-4} (a) for an illustration, where two cases are considered: the first graph is \emph{parallel type}, i.e., 23-endpoint-paths and 24-endpoint-paths are internally disjoint; the second graph is \emph{intersecting type}, i.e. a 23-endpoint-path intersect with a 24-endpoint-paths).  Let $G=G^{C_4}\cup B^4$. We extend $f_0$ to $G$, and the resulting coloring is also denoted by $f_0$. Let  $f'_0$ be the restricted coloring of $f_0$ to $B^4$. Then, $f'_0$ is either the coloring shown in Figure \ref{fig3-4} (b) or its complement coloring with respect to $C_4$; without loss of generality, we assume that  $f'_0$  is the former. For any  $g\in F^{f_0}(G)$, we denote by $g'$ the restricted coloring of $g$ to $G^{C_4}$. 
If $C_4$ is not a bichromatic cycle of $g$, then there must exist a coloring $g'' \in (F^{f_0}(G) \setminus \{g\})$ such that  $C_4$ is a bichromatic cycle of $g''$ and $g''$ contains a bichromatic cycle that intersects with $C_4$. This implies that the restricted coloring of $g''$ to $G^{C_4}$ belongs to $F_2^{f_0}(G^{C_4})$ but is a shared-endpoint coloring on $\{v_1,v_3\}$, a contradiction. Therefore, $|g(C_4)|=2$ and $g' \in F_2^{f_0}(G^{C_4})$.  By the assumption,  $g'$ contains only 23-endpoint-path and 24-endpoint-path. However, $B^4$ contains neither 23-endpoint-path nor 24-endpoint-path. This indicates that $g$ contains no $2i$-cycle $(i=3,4)$ that intersects with $C_4$ in $G$. Additionally, $f'_0$ contains both 13-endpoint-path and 14-endpoint-path in $B^4$ while $g'$ contains no 13-endpoint-path or 14-endpoint-path  in $G^{C_4}$. Therefore, $g$ contains no $1i$-cycle  $(i=3,4)$ that intersects with $C_4$  in $G$. By Theorem \ref{thm2.1}, $G$ is a UBCMPG  with respect to $C_4$. \qed
 \end{proof}

{\bf Remark 5}
For a given 4-base module $G^{C_4}$ with respect to $C_4$, let $f$ be
a module coloring of $G^{C_4}$. In the following, if no specified note, we always assume that a module path of $f$ is a 23-module path or a 24-module path.

According to Theorem \ref{thm3.4}, we have the following corollary.

\begin{corollary}\label{cor3.3}
Suppose that $G^{C_4}$ is a tree-type 4-base module. Then, there exists a coloring  $f_0\in F_2(G^{C_4})$ such that $F^{f_0}(G^{C_4})=\{f_0\}$ and $f_0$ contains a unique 23-module path and a unique 24-module path.
\end{corollary}

\vspace{-0.5cm}

\begin{figure}[H]
  \centering
 \includegraphics[width=10cm]{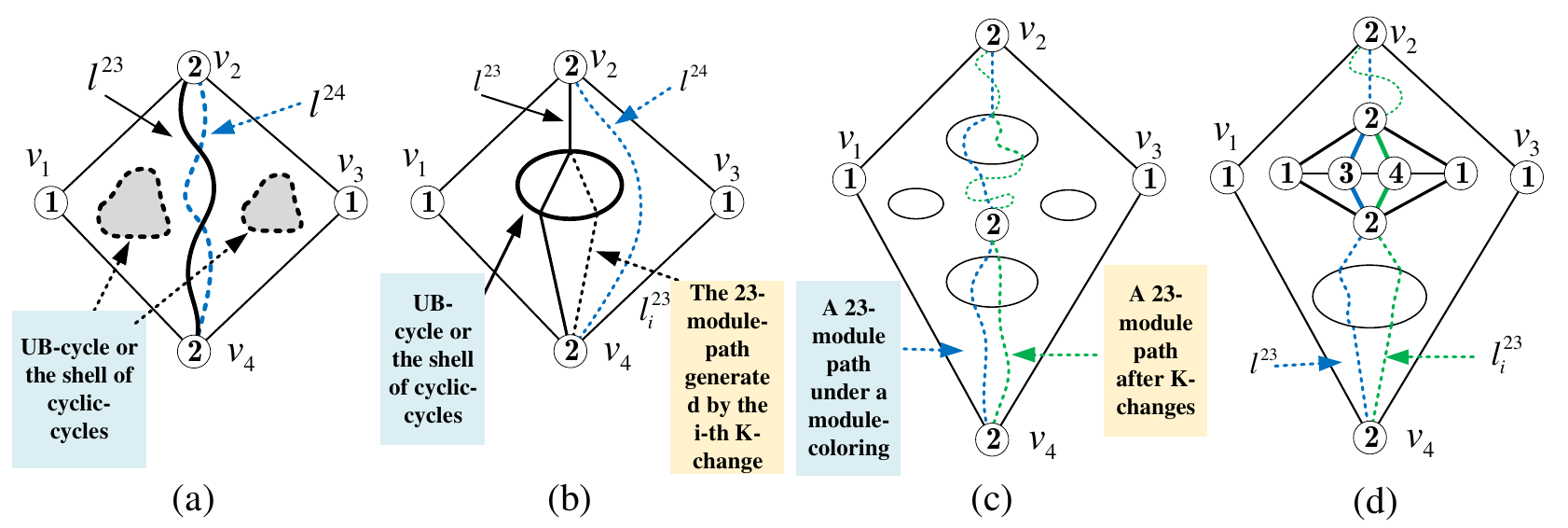}\\
  %\includegraphics[width=11cm]{new1}\\
  %(a) \hspace{3cm} (b)   \hspace{3cm} (c)\hspace{3cm} (d)
  \caption {The structures and classes of cycle-type and cyclic cycle-type 4-base modules. (a)  SMP-type 4-base module, (b) one UB-cycle or one family of cyclic cycles module path, (c) more than one UB-cycle or one family of cyclic cycles module paths, (d) an example of UB-cycle}\label{fig3-5}
\end{figure}

Cycle-type 4-base modules and cyclic cycle-type 4-base modules can further be divided into two classes in terms of module paths: \emph{single module path type} (SMP-type) and \emph{multiple module path type} (MMP-type).

Let $G^{C_4}$ be a cycle-type or cyclic cycle-type 4-base module and $f_0$ be a module coloring of  $G^{C_4}$. If all colorings in $F^{f_0}(G^{C_4})$ have the same 23-module path and 24-module path, denoted by $\ell^{23}$ and $\ell^{24}$, then $G^{C_4}$ is called an SMP-type 4-base module. Clearly, if $G^{C_4}$  is an SMP-type 4-base module, then  no vertex belonging to $\ell^{23}$ or $\ell^{24}$ is in the interior of UB-cycles or cyclic-cycles of $G^{C_4}$; see Figure \ref{fig3-5} (a) for an illustration.

Let $G^{C_4}$ be a cycle-type or cyclic cycle-type 4-base module and $f_0$ be a module coloring of  $G^{C_4}$. $G^{C_4}$ is called an MMP-type 4-base module if the following conditions hold: (1) any two distinct colorings $f_1,f_2\in F_2^{f_0}(G^{C_4})$ only contain module paths $\ell^{23}$ and $\ell^{24}$; (2) $f_2$ is obtained from $f_1$ by conducting one or more $\sigma$-operations, where the UB-cycle or the families of cyclic cycles must contain vertices on  $\ell^{23}$ or $\ell^{24}$ such that the resulting coloring by conducting a $\sigma$-operation is still a module coloring with module paths $\ell^{23}$ and $\ell^{24}$. Figure \ref{fig3-5}(b) depicts the case that there is only one UB-cycle or one family of cyclic cycles;  Figure \ref{fig3-5}(c) presents the case that there are multiple UB-cycles or multiple families of cyclic cycles;  Figure \ref{fig3-5}(d) gives an example of UB-cycle based on Figure \ref{fig3-5}(c).

\section{Contracting and Extending System}
\label{sec:ecsystem}

In \cite{r18}, the author introduced the contracting and extending system (CE-system) $<K_4; \Phi=\{\zeta^+_2,\zeta^+_3,\zeta^+_4,\zeta^+_5, \zeta^-_2,\zeta^-_3,\zeta^-_4,\zeta^-_5\}>$, in which $K_4$ is the starting graph, and $\zeta^+_i$ and $\zeta^-_i$ are a pair of operators (called contracting and extending $i$-wheel operation, respectively), $i=2,3,4,5$. This system builds a connection between colorings and structures when constructing MPGs. Here is a description of these operators, which we will use to prove our subsequent conclusions.

 \begin{figure}[H]
  \centering
 % \includegraphics[width=15cm]{4-1}\\
  \includegraphics[width=6.5cm]{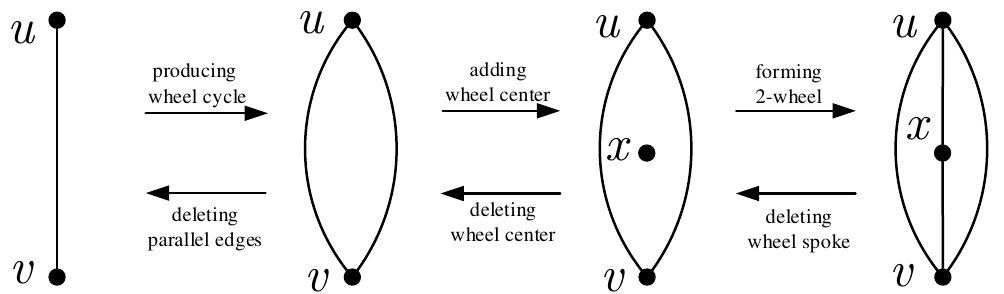} \hspace{0.2cm}
   \includegraphics[width=4cm]{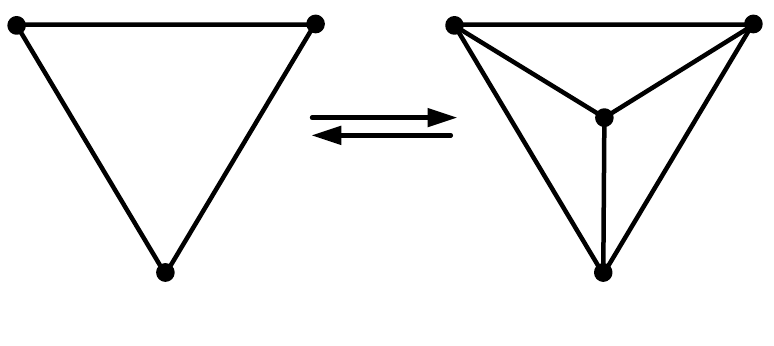}\\
   \hspace{2.3cm}(a)\hspace{6cm} (b)
  \caption {(a) E2WO and C2WO ~~(b) E3WO and C3WO }\label{23-wheel}
\end{figure}

\emph{The extending $2$-wheel operator} (E2WO):  To create a 2-wheel, start by adding a new edge between two adjacent vertices $u,v$, resulting in two parallel edges between $u$ and $v$. Then, add a new vertex $x$  on the face bounded by these parallel edges and connect it to  $u$ and $v$, forming a 2-wheel. This process is illustrated in Figure \ref{23-wheel}(a). \emph{The contracting $2$-wheel operator} (C2WO): To contract a 2-wheel $x$-$uvu$, remove its center vertex  $x$ and the edges $xu$ and $xv$, along with one parallel edge $uv$. We refer to these operations as $\zeta^+_2$ and $\zeta^-_2$, respectively.

\emph{The extending $3$-wheel operator} (E3WO), denoted by $\zeta^+_3$, refers to the transformation from triangle to a 3-wheel, and its inverse process is \emph{the contracting $3$-wheel operator} (C3WO), denoted by $\zeta^-_3$; see Figure \ref{23-wheel}(b).

\begin{figure}[H]
  \centering
 % \includegraphics[width=15cm]{4-1}\\
  \includegraphics[width=11.5cm]{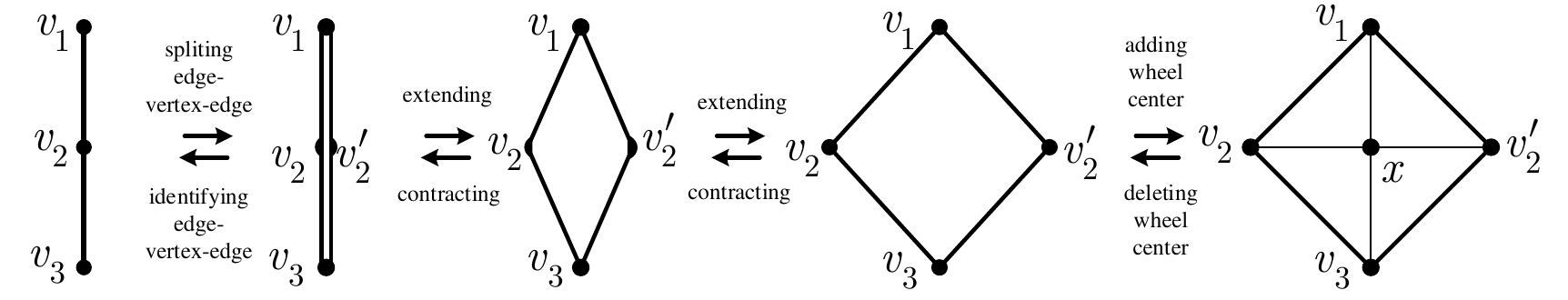}
  \caption {E4WO and C4WO}\label{4-wheel}
\end{figure}

\emph{The extending $4$-wheel operator} (E4WO), denoted by $\zeta^+_4$, refers to the transformation from a path $P_3=v_1v_2v_3$ of length 2 to a 4-wheel, and its inverse is \emph{the contracting $4$-wheel operator} (C4WO), denoted by $\zeta^-_4$; see Figure \ref{4-wheel}.  During the $\zeta^+_4$ process, edge $v_1v_2$, vertex $v_2$, and edge $v_2v_3$ are split into $v_1v_2$ and $v_1v'_2$, $v_2$ and $v'_2$, and $v_2v_3$ and $v'_2v_3$, respectively. Also, edges that are incident with $v_2$ and lie on the left side of $P_3$ (in the original graph) are still incident with $v_2$, while edges that are incident with $v_2$ and lie on the right side (in the original graph) of $P_3$ are now incident with $v'_2$. On the other hand, during the $\zeta^-_4$ process, $v_2$ and $v'_2$ are identified into a new vertex that is incident with all edges that were previously incident with $v_2$ and $v'_2$ in the original graph. $v_2$ and $v'_2$ are referred to as  \emph{contracted vertices} of $\zeta^-_4$.

\begin{figure}[H]
  \centering
 % \includegraphics[width=15cm]{4-1}\\
  \includegraphics[width=12.5cm]{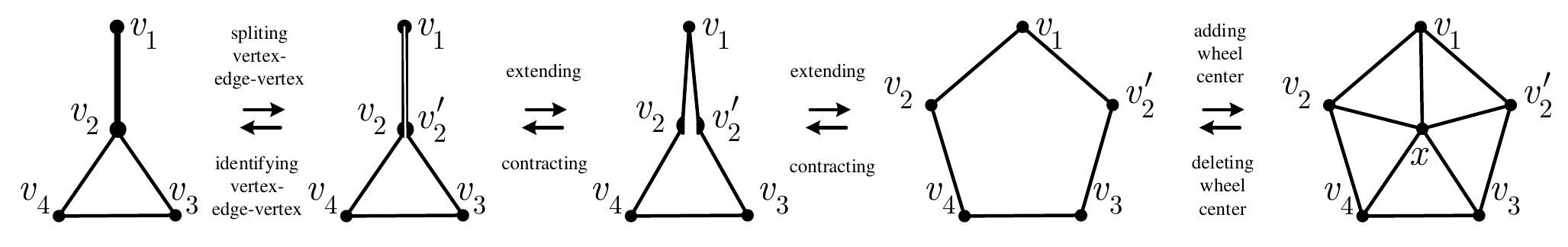}
  \caption {E5WO and C5WO}\label{5-wheel}
\end{figure}

\emph{The extending $5$-wheel operator} (E5WO), denoted by $\zeta^+_5$, transforms a funnel (the first graph from the left in Figure \ref{5-wheel}) into a 5-wheel. Its inverse process is the contracting 5-wheel operator (C5WO), denoted by $\zeta^-_5$ (see Figure \ref{5-wheel}). During the $\zeta^+_5$ process, vertex $v_2$ is split into $v_2$ and $v'_2$, edge $v_1v_2$ into $v_1v_2$ and $v_1v'_2$, and edges that are incident with $v_2$ on the left side of path $v_1v_2v_4$ (in the original graph) become incident with $v_2$, while edges that are incident with $v_2$ on the right side of path $v_1v_2v_3$  become incident with $v'_2$. On the other hand, during the $\zeta^-_5$ process, $v_2$ and $v'_2$ are identified into a new vertex incident with all edges that were previously incident with $v_2$ and $v'_2$ in the original graph.   $v_2$ and $v'_2$ are called \emph{contracted vertices} of $\zeta^-_5$ process.

{\bf Remark 6}. Throughout this paper, we assume that E$i$WO and C$i$WO ($i=2,3,4,5$) are implemented with a given 4-coloring. In the case where $i=4$, the two vertices of the 4-wheel that have been assigned the same color are designated as contracted vertices. Furthermore, if the object of E$i$WO (i=2,3,4,5) is colored with at most three colors based on the 4-coloring, we color the center of the newly formed wheel  (the newly added vertex)  with a color that has not been assigned to the vertices of the wheel cycle, while keeping the colors of other vertices unchanged. Specifically, when $i\in \{4,5\}$, we color $v'_2$ with the color assigned to $v_2$.

\section{Transformation from FCC to The Decycle Problem} \label{sec5}

In this section, we present a method for transforming FCC into the decycle problem, utilizing 4-base modules and a CE-system. We first introduce some fundamental theories of decycle colorings before proceeding to transform FCC into the decycle problem of 4-base modules.

\subsection{Basic theories of decycle colorings}
Let $G^{C_4}$ be a 4-base module and $f$ a module coloring of $G^{C_4}$, where $C_4=v_1v_2v_3v_4v_1$,  $f(v_1)=f(v_3)=1$, $f(v_2)=f(v_4)=2$, and $P_{2i}^f(v_2,v_4)\neq \emptyset$ for $i=3,4$. If there exists a 4-coloring $f^*\in C_4^0(G^{C_4})$ such that $f^*(v_2)\neq f^*(v_4)$, then we call $f^*$ a \emph{decycle coloring} of $G^{C_4}$. If $G^{C_4}$ contains a decycle coloring, then $G^{C_4}$ is said to be  \emph{decyclizable}. As shown in Figure \ref{revised-4-1}, the first two graphs show the diagram of decycle coloring, and the last two graphs describe a module coloring and a decycle coloring of $B^4$. 

\vspace{-0.5cm}
\begin{figure}[H]
  \centering
 \includegraphics[width=8.5cm]{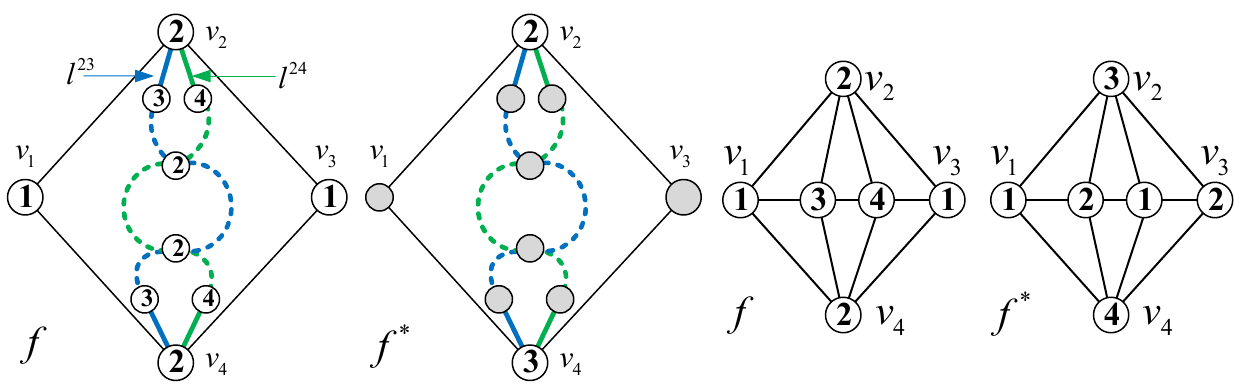}\\
  %\includegraphics[width=11cm]{new1}\\
  %(a) \hspace{3cm} (b)   \hspace{3cm} (c)\hspace{3cm} (d)
  \caption {Illustration of the decycle-coloring}\label{revised-4-1}
\end{figure}

\begin{theorem}\label{thm4-3}
Let $G^{C_4}$ be a 4-base module and $f$ a module coloring, where $C_4=v_1v_2v_3v_4v_1$. If $d_{G^{C_4}}(v_2)=4$ and there are parallel 23-module path $\ell^{23}\in P^f_{23}(v_2,v_4)$ and 24-module path $\ell^{24}\in  P^f_{24}(v_2,v_4)$, then  $G^{C_4}$ is decyclizable, where $f(v_2)=f(v_4)=2$ and $f(v_1)=f(v_3)=1$.
\end{theorem}

\vspace{-0.5cm}
\begin{figure}[H]
  \centering
  \includegraphics[width=8.5cm]{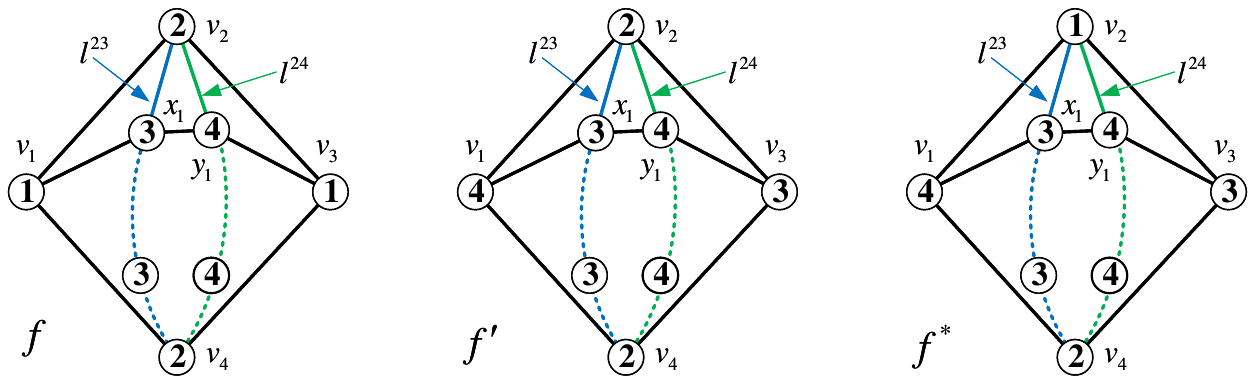}\\
  \caption {Illustration for the proof of Theorem \ref{thm4-3}}\label{figg4-5}
\end{figure}

\begin{proof}
By the definition of $\ell^{23}$ and $\ell^{24}$, $f$ contains at least two 13-components and two 14-components. Note that $G^{C_4}[N_{G^{C_4}}(v_2)]$ is a path of length 3, denoted by $v_1x_1y_1v_3$, where $f(x_1)=3$ and $f(y_1)=4$; see Figure \ref{figg4-5} (a).

Based on $f$, we carry out a $K$-change for the 14-component of $f$ containing $v_1$ and carry out a $K$-change for the 13-component of $f$ containing $v_3$, and denote by $f'$ the resulting coloring;  see Figure \ref{figg4-5} (b). Clearly, $f'_1$ contains  a 34-path $v_1x_1y_1v_3$ from $v_1$ to $v_3$. Therefore, a decycle coloring of $G^{C_4}$ can be obtained by changing the color of $v_2$ from 2 to 1, as shown in Figure \ref{figg4-5} (c). \qed
\end{proof}

According to Theorem \ref{thm4-3}, we can limit our attention to the situation where the module paths $\ell^{23}$ and $\ell^{24}$ intersect under module colorings. This implies that if a module coloring $f$ is used, then any two module paths $\ell^{23}$ and $\ell^{24}$ must have at least three vertices in common, as depicted in Figure \ref{fig3-4} (b). We can also prove that the 4-base module $G^{C_4}$ is decyclizable for this case by utilizing two novel tools, color-connected potential invariant and pocket operations. 

\begin{theorem}\label{decycle-thm}[\textbf{Decycle Theorem}]
Suppose that $G^{C_4}$ is a 4-base module with a module coloring $f$ such that $P^f_{23}(v_2,v_4)\neq \emptyset$ and $P^f_{24}(v_2,v_4)\neq \emptyset$, where $C_4=v_1v_2v_3v_4v_1$. If  $d_{G^{C_4}}(v)\geq 5$ for every $v\in V(G^{C_4})\setminus \{v_1,v_2,v_3,v_4\}$, then there is a 4-coloring $f^* \in C_4^0(G^{C_4})$ satisfying $f^*(v_2) \neq f^*(v_4)$.
\end{theorem}

The proof of the decycle theorem will be presented in the second part of this series of articles. Based on the decycle theorem, a mathematical proof of FCC can be provided. This paper transforms FCC into a decycle problem of  4-base modules, thus termed the ``transformation step for the mathematical proof of FCC.''

The upcoming installment in this article series will showcase the proof of the decycle theorem. Leveraging this theorem, a mathematical proof of FCC can be provided. Specifically, this paper transforms FCC into the decycle problem of 4-base modules,  which represents a crucial stage in realizing the mathematical proof of FCC.

%\subsection{Proof of four color theorem}

\begin{theorem}\label{main}
Every maximal planar graph $G$ with $\delta(G)=5$ is 4-colorable.
\end{theorem}

\begin{proof}
Let us define a color set as $\{1,2,3,4\}$ and proceed to prove the theorem using induction on the order of graphs. It is worth noting that the icosahedron is the smallest MPG with a minimum degree of 5, which has 10 4-colorings. See Figure \ref{ico} for a visual representation of this graph and one of its 4-colorings. We assume that the conclusion holds for every MPG of minimum degree 5 with an order of at most $n$, where $n \geq 12$. We now examine the case where MPGs of minimum degree 5 have an order of $n+1$.

\vspace{-0.5cm}
\begin{figure}[H]
  \centering
  \includegraphics[width=3.2cm]{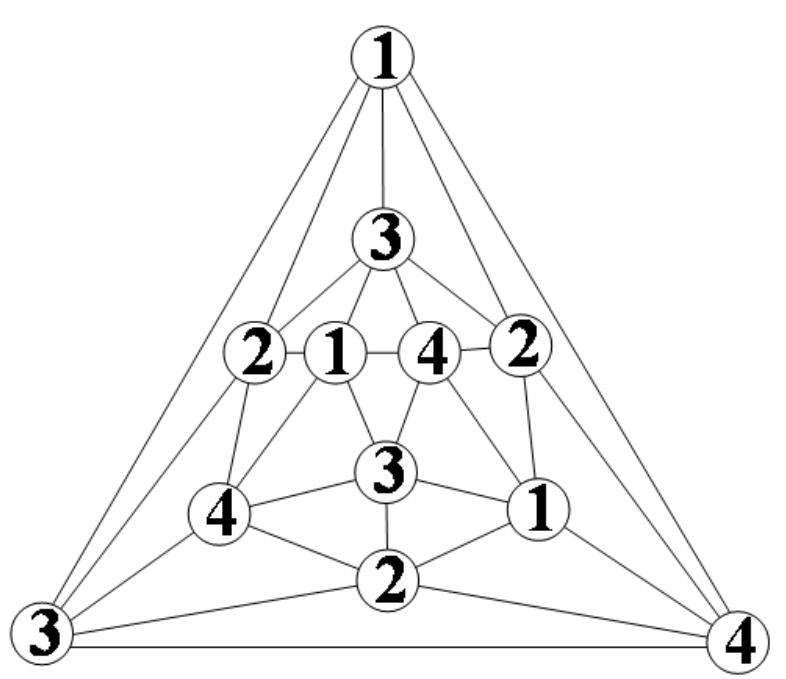}\\
  \caption {The smallest MPG of minimum degree 5 and one of its 4-coloring}\label{ico}
\end{figure}

Let $G$ be an $(n+1)$-order MPG with $\delta(G)=5$, and  $v_2$ an arbitrary vertex of degree $5$. Then, $G[N_G[v_2]]$ is 5-wheel $W_5$. We denote by  $C_5=y_1v_3v_4v_1x_1y_1$ the cycle of $W_5$, i.e., $W_5$=$v_2$-$y_1v_3v_4v_1x_1y_1$; see Figure \ref{figure5-1} (a), in which we use the 5-wheel $W_5$ to simply represent the graph $G$.  By the induction hypothesis, $G-v_2$ is 4-colorable, since $G-v_2$ is a subgraph of some MPG of order $n$. If there exists a 4-coloring $g \in C_4^0(G-v_2)$  such that $|g(C_5)|=3$, then $g$ can be extended to a 4-coloring of $G$ by coloring $v_2$ with the color $\{1,2,3,4\}\setminus g(C_5)$. So, we may assume that  $C_5$ is colored with four colors under every 4-coloring $f$ of $G-v_2$.  Without loss of generality, we assume that $f(v_3)=f(v_1)=1$, $f(v_4)=2$, $f(x_1)=3$, and $f(y_1)=4$; see Figure \ref{figure5-1} (a).

If $v_4$ and $y_1$  are not in the same 14-component of $f$, then we carry out a $K$-change for the 14-component of $f$ containing vertex $v_4$ (or $y_1$) and obtain a new 4-coloring of $G-v_2$, say $f'$. Then, $|f'(C_5)|=3$, which contradicts the above assumption. Therefore, we assume that  $f$ contains a 24-path  from $v_4$ to $y_1$; see Figure \ref{figure5-1} (b).  
%For the same reason, there is also a 23-path from $v_4$ to $x_1$; see Figure \ref{figure5-1} (b).
 Clearly, $v_3$ and $x_1$ do not belong to the same 13-component of $f$. Then, based on $f$, we conduct
a $K$-change for the 13-component containing vertex $v_3$ and obtain a new 4-coloring of $G-v_2$, denoted by $f_1$. We have that $f_1(v_3)=f_1(x_1)=3, f_1(v_1)=1, f_1(v_4)=2$, and $f_1(y_1)=4$; see Figure \ref{figure5-1} (c). By a reverse process, we can also obtain $f$ from $f_1$. We use $`\sim$' to denote this relation between $f$ and $f_1$, i.e., $f\sim f_1$. 
With an analogous discussion as $f$, we see that $f_1$ contains a 14-path  from $v_1$ to $y_1$, and so $x_1$ and $v_4$ do not belong to the same 23-component of $f_1$; see Figure \ref{figure5-1} (d). We then conduct a $K$-change for the 23-component of $f_1$ that contains $x_1$ and obtain a new 4-coloring of $G-v_2$, denoted by $f_2$; see Figure \ref{figure5-1} (e).  Under $f_2$, there exists a 13-path  from $v_1$ to $v_3$, and $v_4$ and $y_1$ do not belong to the same 24-component of $f_2$; see Figure \ref{figure5-1} (f). By conducting a $K$-change for the 24-component of $f_2$ that contains $v_4$, we obtain a new 4-coloring $f_3$ of $G-v_2$; see Figure \ref{figure5-1} (g).
Finally, based on $f_3$, there is a 23-path from $v_3$ to $x_1$, impling that $y_1$ and $v_1$ do not belong to the same 14-component of $f_3$; see Figure \ref{figure5-1} (h). We conduct a $K$-change on the 14-component of $f_3$ that contains $y_1$ and obtain a new 4-coloring $f_4$ of $G-v_2$; see Figure \ref{figure5-1} (i).

\vspace{-0.5cm}
 \begin{figure}[H]
  \centering
  \includegraphics[width=10cm]{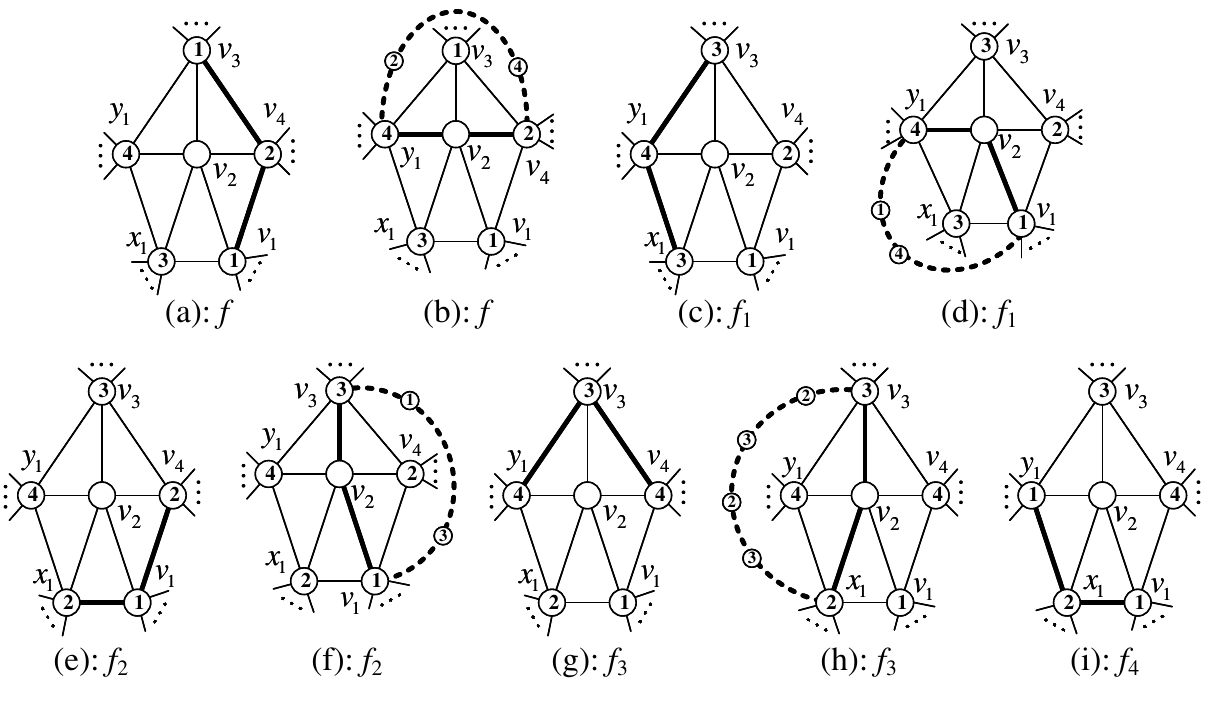}\\
  \caption {The equivalence of the five 4-colorings of $G$ in the condition that $C_5$ are colored with four colors}\label{figure5-1}
\end{figure}

Note that each 4-coloring in $\{f,f_1,f_2,f_3,f_4\}$ has a unique bichromatic 2-path on $C_5$ (marked with bold lines in Figure \ref{figure5-1}), and any two distinct colorings have distinct such bichromatic 2-paths. Therefore, considering only the 4-coloring restricted to $C_5$, there are five unique bichromatic 2-paths. This implies that $`\sim$' is an equivalence relation, and $f\sim f_1 \sim f_2 \sim f_3 \sim f_4 \sim f$.

In 1904, Wernicke \cite{r24} proved that every MPG of minimum degree 5 contains a 55-configuration $G_{5,5}$ or a 56-configuration $G_{5,6}$. Without loss of generality, we assume that the 5-wheel contained in $G_{5,5}$ and $G_{5,6}$ is $W_5$, the 5-wheel shown in Figure \ref{figure5-1} (a).  By the equivalency of 4-colorings in $\{f,f_1,f_2,f_3,f_4\}$, we further assume that under $f$ the bichromatic 2-path of $C_5$ is $v_3v_4v_1$, where $f(v_3)=f(v_1)=1, f(v_4)=2, f(x_1)=3$, and $f(y_1)=4$, as shown in Figures \ref{fig5-2} (a) and (b) for  $G_{5,5}$ and $G_{5,6}$, respectively. The gray vertices in Figures \ref{fig5-2} are colored with some colors under $f$ that can not identified.

Now, we will prove that $f$ can be extended to a 4-coloring of graph $G$. The given proof is for 55-configurations, and the process for 56-configurations is similar and thus omitted.

Suppose that $G$ contains $G_{5,5}$. We conduct an E4WO on the bichromatic 2-path $v_3v_4v_1$, and the resulting graph is shown in Figure \ref{fig5-2} (c), denoted by $G^*$, where $v'_4$ and $v_4$ are the two vertices replacing $v_4$ and $x$ is the center of the newly 4-wheel $x-v_3v_4v_1v'_4v_3$. Clearly, $G^*$ is an MPG of order $n+3$. 
Based on $f$, there is a 4-coloring $f'$ of  $G^*-v_2$ such that $f'(v)=f(v)$ for every $v\in V(G^*)\setminus \{v'_4,v_4,x,v_2\}$, $f'(v_4)=f'(v'_4)=2$, and  $f'(x)=3$ or $4$, where $v_2$ is the unique vertex in $G^*$ that is not assigned to a color under $f'$ (see Figure \ref{fig5-2} (c)). 
Furthermore, based on $f'$ we can obtain a 4-coloring $f''$ of $G^*$ by changing the colors of vertices in the 5-cycle $v_3v_4v_1x_1y_1v_3$, i.e., recoloring $v'_4$ with 3,  coloring $v_2$ with 2, and  recoloring $x$ with 4; see Figure \ref{fig5-2} (d). Observe that $f''(v'_4)\neq f''(v_4)$.  We are currently facing a problem on how to use $f''$ to generate a 4-coloring of $G^*$, denoted by $f'''$, such that  $f'''(v'_4)=f'''(v_4)$. Thus, based on $f'''$, by conducting a C4WO on the 4-wheel $x-v_3v_4v_1v'_4v_3$, we can obtain a 4-coloring of $G$.  In order to accomplish this, it is necessary to thoroughly analyze the attributes of $G^*$.

\vspace{-0.5cm}
\begin{figure}[H]
  \centering
  \includegraphics[width=10cm]{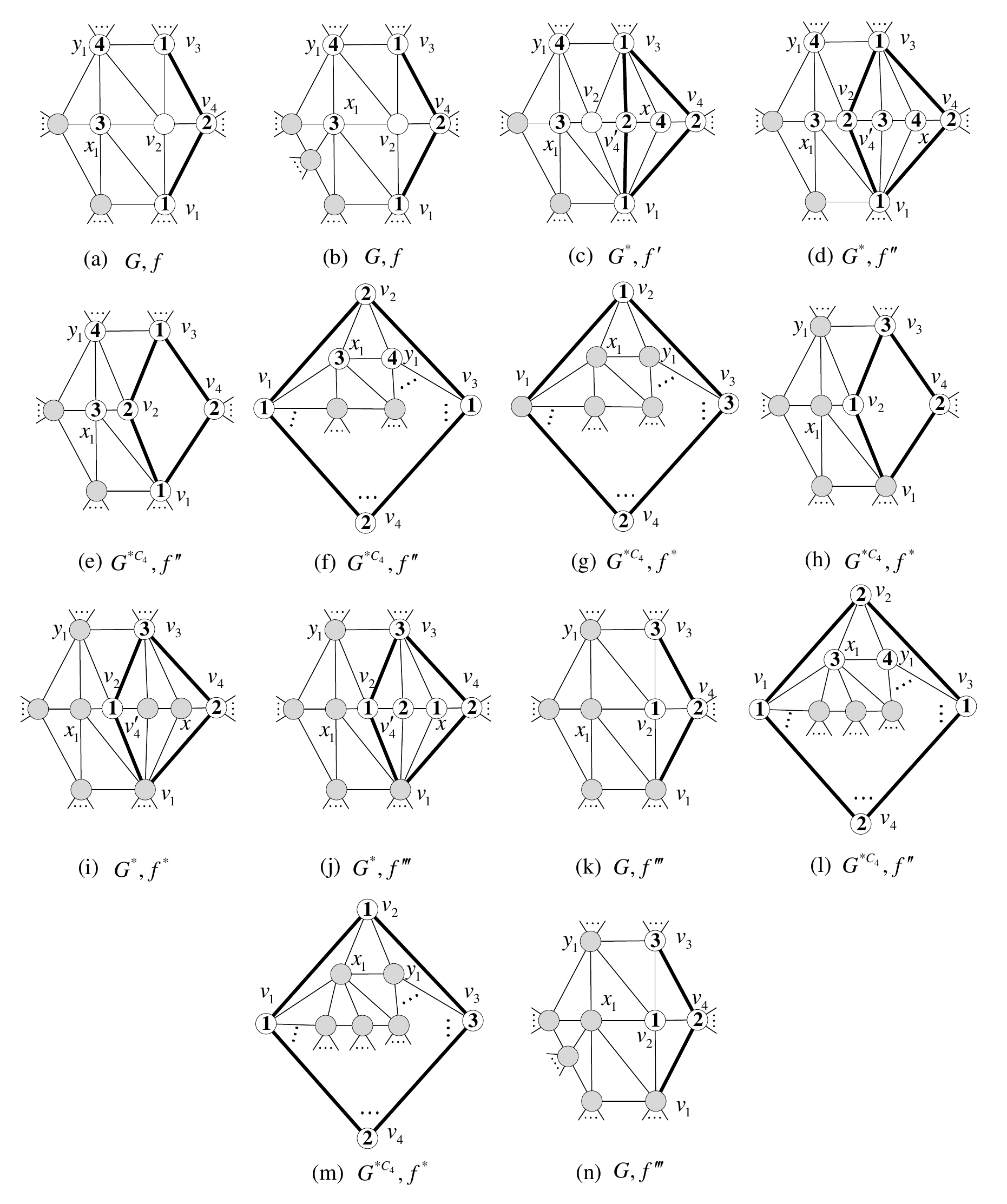}\\
  \caption {An illustration diagram of the proof of Theorem \ref{main}}\label{fig5-2}
\end{figure}

Note that the SMPG consisting of the 4-cycle $C_4=v_1v_2v_3v_4v_1$ and its interior vertices is the smallest 4-base module $B^4$. Moreover, $G^{*C_4}=G^*-\{x,v'_4\}$ is a 4-chromatic SMPG, and the restricted coloring of $f''$ to $G^{*C_4}$ is also dentoed by $f''$, as shown in Figure \ref{fig5-2} (e). By means of topological transformation, $G^{*C_4}$ can be plane embedded with $C_4$ as the outer face (unbounded face). Please refer to Figure \ref{fig5-2} (f) for a visual representation of the embedding. Clearly, $d_{G^{*C_4}}(v_2)=4$ and $d_{G^{*C_4}}(v_3)=5$. We consider two cases.

First, $G^{*C_4}$ is not a 4-base module. By Theorem \ref{thm3.4}, there exists a $f^*\in C_4^0(G^{*C_4})$ such that $f^*(v_2)\neq f^*(v_4)$.  

Second,  $G^{*C_4}$ is a 4-base module. By Theorem \ref{decycle-thm} (decycle theorem),   there also exists a $f^*\in C_4^0(G^{*C_4})$ such that $f^*(v_2)\neq f^*(v_4)$.  

As a result, there always exists a 4-coloring  $f^*\in C_4^0(G^{*C_4})$ such that $f^*(v_2)\neq f^*(v_4)$.  
Without loss of generality, we assume that $f^*(v_2)=1, f^*(v_4)=2$, and  $f^*(v_3)=3$, as shown in  Figure \ref{fig5-2} (g).

Now, by applying a topological transformation to $G^{*C_4}$, we restore $G^{*C_4}$ to its original structure displayed in Figure \ref{fig5-2} (e). The resulting graph after transformation is depicted in Figure \ref{fig5-2} (h), where the current coloring is $f^*$. To convert $G^{*C_4}$ back to $G^*$, we add vertices $v'_4$ and $x$ inside $C_4$ and connect them to vertices $v_1$, $v_2$, $v_3$, and $v_4$, as demonstrated in Figure \ref{fig5-2} (i). Notably, $v'_4$ and $x$ have pending coloring under $f^*$.  By coloring $v'_4$ with color 2 and $x$ with color 1, $f^*$ is extended to a 4-coloring of $G^*$, denoted as $f'''$, as shown in Figure \ref{fig5-2} (j). It is worth noting that $f'''(v'_4)$ and $f'''(v_4)$ are always 2, regardless of whether $f'''(v_1)$ is 3 or 4. Therefore, performing a C4WO on the 4-wheel $x-v_3v_4v_1v'_4v_3$ in $G^*$ yields the original MPG $G$. The restricted coloring of $f'''$ to $G$, denoted also as $f'''$, is a 4-coloring of $G$, as shown in Figure \ref{fig5-2} (k). This establishes the theorem for the case where $G$ contains a 55-configuration (Figure \ref{fig5-2} (a)).
For the situation where $G$ comprises a 56-configuration (Figure \ref{fig5-2} (b)), a similar proof can be presented. The analogous $G^{*C_4}$ and $f''$, $G^{*C_4}$ and $f^8$, and $G^*$ and $f^*$ for this situation are illustrated in Figures \ref{fig5-2} (l), (m), and (n). This completes the proof. \qed

\end{proof}

In order to further elucidate our method used to prove FCC, we take the well-known Heawood counterexample $G$ as an example to show the process of our proof in Theorem \ref{main}; see Figure \ref{fig5-3} (a), where$v_2$ is a vertex of degree 5 and $f$ is a 4-coloring of $G-v_2$. Then, $v_2$ is the unique vertex whose color is unfixed under $f$.  $G[N_G[v_2]]$ is a 5-wheel $W_5$ with cycle $C_5=y_1y_3v_4v_1x_1y_1$, where $f(v_3)=f(v_1)=1$, $f(v_4)=4$, $f(y_1)=3$, $f(x_1)=2$. Based on $f$, the bichromatic 2-path on $C_5$ is $v_3v_4v_1$. By conducting an E4WO on $v_3v_4v_1$ we obtain a graph $G^*$ with a natural coloring $f'$; see Figure \ref{fig5-3}(b). 
In  $G^*$, $v_4$ and $v'_4$ are obtained by splitting $v_4$ in $G$, and $x$ is the certer of the newly 4-wheel. Based on $f'$, a 4-coloring $f''$ of $G^*$ is obtained by coloring $v_2$ with 4 and $v'_4$ with 2, as shown in Figure \ref{fig5-3} (c). 
Now,  delete the two vertices $v'_4$ and $x$ from $G^*$, and a   4-chromatic SMPG $G^{*C_4}$ is obtained.  The restricted coloring of $f''$ to $G^{*C_4}$ is still denoted by $f''$, as shown in Figure \ref{fig5-3} (d). It is clear to see that $G^{*C_4}$ is a 4-base module and its module paths are between $v_2$ to $v_4$. According to Theorem \ref{decycle-thm} (decycle theorem ), there is a decycle coloring $f^*$ of $G^{*C_4}$, as shown in Figure \ref{fig5-3} (e). Furthermore, to convert $G^{*C_4}$ back to $G^*$,  we add vertices  $v'_4$ and $x$ inside $C_4$ and connect them to vertices $v_1$, $v_2$, $v_3$, and $v_4$, as demonstrated in Figure \ref{fig5-3} (f). The extended coloring of $f^*$ to $G^*$ is still denoted by  $f^*$, where $v'_4$ and $x$ are vertices with pending coloring. Clearly, by coloring $v'_4$ with color 4 and $x$ with color 3, $f^*$ is extended to a 4-coloring of $G^*$, denoted as $f'''$, as shown in Figure \ref{fig5-3} (g).  Since $f'''(v'_4)=f'''(v_4)=4$, we can obtain the original MPG $G$ by performing a C4WO on the 4-wheel $x-v_3v_4v_1v'_4v_3$ in $G^*$. The restricted coloring of $f'''$ to $G$, denoted also as $f'''$, is a 4-coloring of $G$, as shown in Figure \ref{fig5-3} (h).

\begin{figure}[H]
  \centering
  \includegraphics[width=10cm]{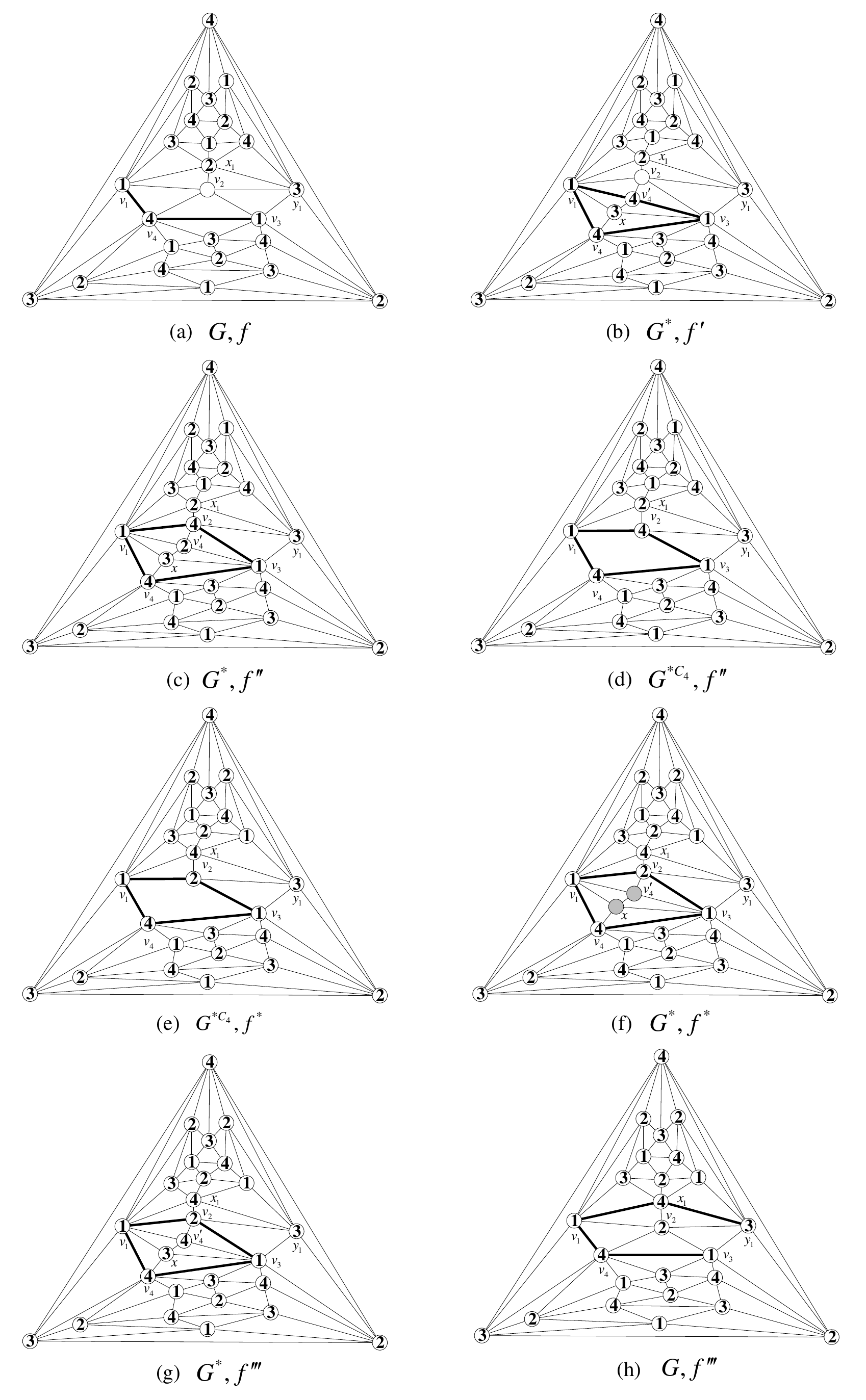}\\
  \caption {The process of coloring the Heawood counterexample by Theorem \ref{main}}\label{fig5-3}
\end{figure}

\textbf{Remark 7} In the above example, please refer to the second part of this series of articles to obtain a decycling coloring of $G^{*C_4}$.

\textbf{Remark 8}  Wernicke's findings \cite{r24} show that a maximal planar graph with a minimum degree of 5 must possess edges of two types: both endpoints have a degree of 5, or one endpoint has a degree of 5, and the other has a degree of 6. Borodin later expanded on this in 1992 \cite{r25}, demonstrating that in the set of triangles of such a graph, there must be one of four types of triangles. These are: each vertex has a degree of 5; two vertices have a degree of 5, and one has a degree of 6; two vertices have a degree of 5, and one has a degree of 7; or two vertices have a degree of 6, and one has a degree of 5. These configurations are respectively referred to as 555-configuration, 556-configuration, 557-configuration, and 566-configuration (since the subgraph induced by the three vertices of a triangle and their neighborhood forms a configuration). This paper specifically focuses on these four types of configurations and, as a result, the degree of $v_2$ is limited to 5, the degree of $x_1$ is limited to 5 or 6, and the degree of $y_1$ is limited to 5, 6, or 7 in general.

\end{document}